\documentclass[11pt,reqno]{amsart}

\usepackage[margin=1.5in]{geometry}
\usepackage[colorinlistoftodos]{todonotes}

\newcounter{todocounter}

\usepackage[T1]{fontenc}
\usepackage{tgpagella}
\usepackage{mathpazo}

\usepackage[pagebackref]{hyperref} 
\usepackage{xcolor}
\usepackage{amsmath,amsthm}
\usepackage{amssymb}
\usepackage{graphicx}
\usepackage[mathscr]{eucal}
\usepackage{MnSymbol}
\usepackage[frame,ps,matrix,arrow,curve,rotate,all,2cell,tips]{xy}
\usepackage{epic,eepic}
\usepackage{enumerate}
\usepackage{color,soul}

\newtheorem{theorem}[subsection]{Theorem}
\newtheorem{lemma}[subsection]{Lemma}

\newtheorem{proposition}[subsection]{Proposition}
\newtheorem{corollary}[subsection]{Corollary}

\theoremstyle{definition}
\newtheorem{definition}[subsection]{Definition}
\newtheorem{example}[subsection]{Example}
\newtheorem{remark}[subsection]{Remark}
\newtheorem{assumption}[subsection]{Assumption}

\numberwithin{equation}{subsection}

\renewcommand{\O}{\mathsf{O}}

\newcommand{\M}{\mathcal{M}}

\newcommand{\Mtoc}{\M^{\fC}}
\newcommand{\msigmanop}{\M^{\sigmaop_n}}
\newcommand{\sigman}{\Sigma_n}

\newcommand{\Mdelta}{\M^{\Delta}}

\newcommand{\rkm}{R_K \M}
\newcommand{\rkmsigmanop}{(\rkm)^{\sigmaop_n}}
\newcommand{\rkmc}{(\rkm)^{\fC}}
\newcommand{\rkmdelta}{(\rkm)^{\Delta}}

\newcommand{\C}{\mathcal{C}}

\newcommand{\D}{\mathcal{D}}

\newcommand{\sF}{\mathscr{F}}

\newcommand{\sI}{\mathscr{I}}
\newcommand{\sJ}{\mathscr{J}}

\renewcommand{\emptyset}{\varnothing}

\renewcommand{\tilde}[1]{\widetilde{#1}}

\DeclareMathOperator{\colim}{colim}

\renewcommand{\hom}{\operatorname{Hom}}
\DeclareMathOperator{\Ho}{Ho}

\newcommand{\po}{\ar@{}[dr]|(.7){\Searrow}}
\newcommand{\pb}{\ar@{}[dr]|(.3){\Nwarrow}}

\newcommand{\Map}{\mathsf{Map}}
\DeclareMathOperator{\map}{map}
\newcommand{\Ch}{\mathsf{Ch}}

\newcommand{\cat}[1]{\mathcal{#1}}

\newcommand{\boxprod}{\mathbin\square}

\usepackage{tikz}
\usetikzlibrary{arrows,decorations.pathmorphing}
\usetikzlibrary{backgrounds,positioning}
\usetikzlibrary{fit,petri,shapes.misc}

\tikzset{auto}

\tikzset{empty/.style={circle,inner sep=0pt,minimum size=6mm}}
\tikzset{emptyvt/.style={circle,inner sep=0pt,minimum size=0mm}}

\tikzset{plain/.style={circle,draw,very thick,
inner sep=0pt,minimum size=6mm}}

\tikzset{fatplain/.style={rounded rectangle,draw,very thick,minimum size=6mm}}

\tikzset{bigplain/.style={rounded rectangle,draw,very thick,minimum size=.8cm}}

\tikzset{yellowvt/.style={circle,draw,fill=yellow,very thick,inner sep=0pt,minimum size=6mm}}

\tikzset{bluevt/.style={circle,draw,fill=blue!20,very thick,inner sep=0pt,minimum size=6mm}}

\tikzset{greenvt/.style={circle,draw,fill=green!30,very thick,inner sep=0pt,minimum size=6mm}}

\tikzset{redvt/.style={circle,draw,fill=red!30,very thick,inner sep=0pt,minimum size=6mm}}

\tikzset{arrow/.style={->,thick}}
\tikzset{dashedarrow/.style={->,dashed,thick}}
\tikzset{dottedarrow/.style={->,dotted,thick}}
\tikzset{mapto/.style={|->,thick}}

\tikzset{implies/.style={thick,double,double equal sign distance,-implies}}

\tikzset{line/.style={thick}}
\tikzset{dottedline/.style={dotted,thick}}
\tikzset{dashedline/.style={dashed,thick}}

\tikzset{inputleg/.style={<-,thick}}
\tikzset{outputleg/.style={->,thick}}
\tikzset{dottedinput/.style={<-,dotted,thick}}

\newcommand{\adjoint}{
\nicearrow\xymatrix{ \ar@<2pt>[r] & \ar@<2pt>[l]}}
\renewcommand{\hookrightarrow}{\nicexy{\ar@{^{(}->}[r] &}}
\newcommand{\nicearrow}{\SelectTips{cm}{10}}
\newcommand{\nicexy}{\nicearrow\xymatrix@C+5pt}

\newcommand{\pushout}{\ar@{}[dr]|(0.75){\Searrow}}
\newcommand{\drrpushout}{\ar@{}[drr]|(0.90){\Searrow}}

\renewcommand{\to}{\hspace{-.1cm}\nicearrow\xymatrix@C-.2cm{\ar[r]&}\hspace{-.1cm}}

\newcommand{\comp}{\circ}
\newcommand{\defn}{\,\overset{\mathrm{def}}{=\joinrel=}\,}

\newcommand{\tensorover}[1]{\underset{#1}{\otimes}}
\newcommand{\timesover}[1]{\underset{#1}{\times}}

\newcommand{\psharp}{p^\#}

\newcommand{\bA}{\mathbf{A}}

\newcommand{\fB}{\mathfrak{B}}
\newcommand{\frakC}{\mathfrak{C}}
\newcommand{\fC}{\mathfrak{C}}

\newcommand{\fD}{\mathfrak{D}}

\newcommand{\sC}{\mathsf{C}}

\renewcommand{\sF}{\mathsf{F}}

\renewcommand{\sI}{\mathsf{I}}
\renewcommand{\sJ}{\mathsf{J}}

\newcommand{\sO}{\mathsf{O}}
\newcommand{\fO}{F_{\sO}}

\newcommand{\Path}{\mathsf{Path}}

\newcommand{\scrI}{\mathscr{I}}

\newcommand{\ua}{\underline{a}}
\newcommand{\ub}{\underline{b}}
\newcommand{\uc}{\underline{c}}

\newcommand{\smallop}{{\scalebox{.5}{$\mathrm{op}$}}}
\newcommand{\cof}{{\scalebox{.5}{$\mathrm{cof}$}}}
\newcommand{\acof}{{\scalebox{.5}{$\mathrm{t.cof}$}}}

\newcommand{\clubcof}{\clubsuit_{\cof}}
\newcommand{\clubacof}{\clubsuit_{\acof}}

\newcommand{\clubacofm}{\clubacof^{\M}}

\newcommand{\clubm}{\clubsuit^{\M}}
\newcommand{\clubk}{\clubsuit^{\rkm}}
\newcommand{\clubcofk}{\clubcof^{\rkm}}
\newcommand{\clubacofk}{\clubacof^{\rkm}}

\newcommand{\cala}{\mathcal{A}}

\newcommand{\calm}{\mathcal{M}}

\newcommand{\Top}{\mathsf{Top}}
\newcommand{\Sp}{\mathsf{Sp}}

\newcommand{\seq}{\mathsf{Seq}}
\newcommand{\seqc}{\seq_{\fC}}
\newcommand{\seqcc}{\seqc(\C)}

\newcommand{\symseq}{\mathsf{SymSeq}}
\newcommand{\symseqc}{\symseq_{\fC}}
\newcommand{\symseqcm}{\symseqc(\calm)}
\newcommand{\symseqcrkm}{\symseqc(R_K\M)}

\newcommand{\alg}{\mathsf{Alg}}

\newcommand{\algo}{\alg(\sO)}
\newcommand{\algom}{\alg(\sO;\M)}
\newcommand{\algorkm}{\alg(\sO;\rkm)}
\newcommand{\algoc}{\alg(\sO;\C)}

\newcommand{\algtm}{\alg(T;\M)}
\newcommand{\algtrkm}{\alg(T;\rkm)}

\newcommand{\gop}{G^{\smallop}}

\newcommand{\sigmab}{\Sigma_{[\ub]}}
\newcommand{\sigmac}{\Sigma_{[\uc]}}

\newcommand{\sigmaop}{\Sigma^{\smallop}}

\newcommand{\sigmaopb}{\Sigma^{\smallop}_{\smallbrb}}
\newcommand{\sigmaopc}{\Sigma^{\smallop}_{\smallbrc}}

\newcommand{\ctothec}{\C^{\fC}}

\newcommand{\operad}{\mathsf{Operad}}

\newcommand{\operadsigmac}{\operad^{\sigmaofc}}
\newcommand{\operadsigmacc}{\operadsigmac(\C)}

\newcommand{\omegac}{\Omega_{\fC}}

\newcommand{\operadomegac}{\operad^{\omegac}}
\newcommand{\operadomegacc}{\operadomegac(\C)}

\newcommand{\pofc}{\Sigma_{\frakC}}
\newcommand{\pofcop}
{\pofc^{\scalebox{.6}{$\mathrm{op}$}}}

\newcommand{\prof}{\mathsf{Prof}}
\newcommand{\profc}{\prof(\fC)}
\newcommand{\profcc}{\profc \times \fC}

\renewcommand{\pb}{\mathcal{P}(\fB)}

\newcommand{\smallprof}[1]
{\raisebox{.05cm}{\scalebox{0.8}{#1}}}

\newcommand{\smallbinom}[2]
{\raisebox{.05cm}{\scalebox{0.8}{$\binom{#1}{#2}$}}}

\newcommand{\duc}
{\smallprof{$\binom{d}{\uc}$}}

\newcommand{\singledbrc}
{\smallprof{$\binom{d}{[\uc]}$}}

\newcommand{\smallbr}[1]
{\raisebox{.03cm}{\scalebox{0.5}{#1}}}

\newcommand{\smallbrb}{\smallbr{$[\ub]$}}

\newcommand{\smallbrc}{\smallbr{$[\uc]$}}

\newcommand{\sigmabrc}{\Sigma_{\smallbr{$[\uc]$}}}

\newcommand{\sigmabrcop}{\sigmabrc^{\smallop}}

\newcommand{\sigmabrcopd}{\sigmabrcop \times \{d\}}

\newcommand{\sigmaofc}{\pofc}
\newcommand{\sigmacop}{\pofcop}
\newcommand{\sigmacopc}{\sigmacop \times \fC}

\newcommand{\impliesspace}{\quad\text{implies}\quad}

\renewcommand{\lim}{\mathsf{lim}\,}

\DeclareMathOperator*{\hocolim}{\mathsf{hocolim}\,}

\DeclareMathOperator{\Hom}{Hom}

\DeclareMathOperator{\Kan}{\mathsf{Kan}}

\newcommand{\deltan}{\Delta[n]}
\newcommand{\ddeltan}{\partial\Delta[n]}

\newcommand{\Lambdak}{\Lambda(K)}
\newcommand{\Lambdakbar}{\overline{\Lambdak}}

\newcommand{\tensoroversigman}{\otimes_{\Sigma_n}}
\newcommand{\boxn}{\boxprod n}
\renewcommand{\diamond}{\blacklozenge}

\newcommand{\staro}{\filledstar^{\sO}}
\newcommand{\starx}{\filledstar^X}

\newcommand{\tensorunit}{I}
\newcommand{\symseqcc}{\symseqc(\C)}

\newcommand{\sigmainv}{\sigma^{-1}}

\begin{document}

\title{Right Bousfield Localization and Operadic Algebras}

\author{David White}
\address{\noindent Department of Math and Computer Science \\ Denison University
\\ Granville, OH}
\email{david.white@denison.edu}

\author{Donald Yau}
\address{\noindent Department of Mathematics \\ The Ohio State University at Newark \\ Newark, OH}
\email{dyau@math.osu.edu}

\begin{abstract}
It is well known that under some general conditions right Bousfield localization exists.  We provide general conditions under which right Bousfield localization yields a monoidal model category.  Then we address the questions of when this monoidal model structure on a right Bousfield localization induces a model structure on the category of algebras over a colored operad and when a right Bousfield localization preserves colored operadic algebras. We give numerous applications, to topological spaces, equivariant spaces, chain complexes, stable module categories, and to the category of small categories. We recover a wide range of classical results as special cases of our theory, and prove several new preservation results.
\end{abstract}

\maketitle

\tableofcontents

\section{Introduction}

The CW Approximation Theorem is a fundamental result in homotopy theory. It allows us to work with CW complexes without losing information, up to weak homotopy equivalence. This result is one of a suite of similar results, and all are examples of right Bousfield localization \cite{hirschhorn}, also called cellularization and colocalization. Examples include $A$-cellular homotopy theory in topological spaces or simplicial sets \cite{chacholski-thesis}, $n$-connected covers and Postnikov pieces \cite{nofech}, analogous constructions in the category of small categories and in the category of simplicial abelian groups, point-set models in chain complexes and $R$-modules for localizing subcategories in the derived category of $R$ and the stable module category of $R$, and family model structures in equivariant homotopy theory (see section \ref{sec:app-equivar-spaces}). Right Bousfield localization also has applications to homotopy limits of left Quillen presheaves, and to constructing Postnikov towers in simplicial or spectral model categories \cite{barwickSemi}.

While much work has been done to understand how much structure is preserved by cellularization \cite{castellana-crespo-scherer-conn-covers}, \cite{mcgibbon-moller}, \cite{chacholski-parent-stanley}, \cite{nofech}, little is understood. In this paper, we find conditions under which right Bousfield localization preserves algebraic structure as encoded by a colored operad. We then prove that these new conditions are satisfied in the examples of interest, and provide specific preservation results about commonly used operads.

Operads are used to encode algebraic structure in general symmetric monoidal categories, and hence have become central to modern algebraic topology. Operads have also been applied to deformation theory and mathematical physics \cite{mss}, in gauge theory and symplectic geometry, in representation theory and graph cohomology \cite{fresse}, and in Goodwillie calculus.

In recent years, the importance of \textit{colored} operads has become clear, e.g. in \cite{bm07}, \cite{jy2}, and \cite{batanin-berger}. Colored operads encode even more general algebraic structures, including the category of operads itself, other categories which encode algebraic structure (e.g. modular operads, higher operads, colored operads), morphisms between algebras over an operad, modules over an operad, other enriched categories, and diagrams in such categories. Colored operads have been applied in enriched category theory, factorization homology, higher category theory (leading to $\infty$-operads), and topological quantum field theories.

Our setting in this paper is a monoidal model category $\M$, a set of objects $K$ (the cells we want to build things out of), and a right Bousfield localization $\rkm$. This means $id:\M \to \rkm$ is right Quillen and $\rkm$ satisfies the universal property that any right Quillen functor $F:\M \to \cat{N}$, taking the $K$-colocal equivalences to weak equivalences, factors through $\rkm$. As a category, $\rkm$ is the same as $\M$, hence is monoidal, but the model structure is different (with more weak equivalences), and it is not automatic that $\rkm$ will be a monoidal model category; i.e., the pushout product axiom could fail in $\rkm$. In order to prove our preservation results, we will need to bring homotopy into the realm of colored operads by building model structures and semi-model structures on categories of algebras over operads. Such structures provide a powerful computational tool which has been crucial in many of the applications above, and are therefore of independent interest. 

In Section \ref{sec:colored-operads} we provide definitions and notations concerning colored operads. In Section \ref{sec:right-bous} we review right Bousfield localization. We assume the reader is familiar with the basics of model categories (an excellent overview is \cite{hovey}), and we encourage the reader to proceed with a copy of \cite{hirschhorn} near at hand. In Section \ref{sec:monoidal} we define the notion of a monoidal right Bousfield localization, dualizing \cite{white-localization}, and we provide an easy to check condition guaranteeing that a right Bousfield localization $\rkm$ satisfies the pushout product axiom. In Section \ref{sec:model-on-algebras} we build our model structures on categories of algebras taken in $\rkm$. In Section \ref{sec:preservation} we dualize our general preservation results from \cite{white-yau}.  In Sections \ref{sec:colocal-colored-operads} and \ref{sec:cofibrant-operads} we derive specific preservation results, often with easier to check hypotheses. Finally, in Sections \ref{sec:applications}--\ref{app:small-cat} we apply our results to numerous examples of interest, including spaces, chain complexes, $R$-modules, small categories, and equivariant spaces.

En route to these examples, we establish many results of independent interest. For example, in Theorem \ref{thm:oalg-g-spaces-model}, we prove that algebras over any colored operad in $G$-spaces have a transferred model structure. In Corollary \ref{cor:oalg-g-spectra-model}, we prove the same for positive model structures on orthogonal $G$-spectra. To our knowledge, these results have not appeared before, and are essential to the study of brave new equivariant algebra. Similarly, in Section \ref{app:stmod}, we conduct what we believe to be the first ever study of operad algebras in the stable module category. This has the potential to give powerful new tools to representation theorists using the stable module category.

\textbf{Acknowledgments:} The authors are grateful to Sarah Wolff for helpful conversations as Section \ref{app:stmod} was being worked out, and to Sinem Odaba\c{s}i for a helpful conversation related to Section \ref{subsec:cotorsion}.

\section{Colored Operads}
\label{sec:colored-operads}

\begin{assumption}
\label{basicassumption}
Fix a symmetric monoidal closed category $(\C, \otimes, \tensorunit, \Hom)$ with all small limits and colimits, initial object $\varnothing$, and terminal object $*$.
\end{assumption}

Let us first recall some notations regarding colors from \cite{jy2}.

\begin{definition}
\label{profiles}
Fix a non-empty set $\fC$ once and for all, whose elements are called \emph{colors}.
\begin{enumerate}
\item
A \emph{$\fC$-profile} is a finite sequence of elements $\uc = (c_1, \ldots, c_m)$ in $\fC$.  The empty $\fC$-profile is denoted $\emptyset$.  Write $|\uc|=m$ for the \emph{length} of a profile $\uc$.
\item
The set of all $\fC$-profiles is denoted by $\profc$.
\item
If $\ua = (a_1,\ldots,a_m)$ and $\ub$ are $\fC$-profiles, then a \emph{left permutation} $\sigma : \ua \to \ub$ is a permutation $\sigma \in \Sigma_{m}$ such that
\[
\sigma\ua = (a_{\sigmainv(1)}, \ldots , a_{\sigmainv(m)}) = \ub
\]
This necessarily implies $|\ua| = |\ub| = m$.  A left permutation is also called a \emph{map}.
\item
The \emph{groupoid of $\fC$-profiles}, with left permutations as the isomorphisms, is denoted by $\pofc$.  The opposite groupoid $\pofcop$ is regarded as the groupoid of $\fC$-profiles with \emph{right permutations}
\[
\ua\sigma = (a_{\sigma(1)}, \ldots , a_{\sigma(m)})
\]
as isomorphisms.
\item
The \emph{orbit} of a profile $\ua$ is denoted by $[\ua]$.  The maximal connected sub-groupoid of $\pofc$ containing $\ua$ is written as $\Sigma_{[\ua]}$.  Its objects are the left permutations of $\ua$.  There is a decomposition
\begin{equation}
\label{pofcdecomp}
\pofc \cong \coprod_{[\ua] \in \pofc} \Sigma_{[\ua]},
\end{equation}
where there is one coproduct summand for each orbit $[\ua]$ of a $\fC$-profile.
\item
A \emph{$\fC$-colored object in $\C$} is an object in the product category 
\[
\ctothec \defn \prod_{\fC} \C.
\]
A typical $\fC$-colored object $X$ is also written as $\{X_a\}$ with $X_a \in \C$ for each color $a$.
\item
A $\fC$-colored object $X$ is said to be \emph{concentrated at $c \in \fC$} if $X_d = \varnothing$ for all colors $d \not= c$.  A map of $\fC$-colored objects is said to be \emph{concentrated at $c \in \fC$} if both its domain and codomain are concentrated at $c$.
\item
The category of \emph{$\fC$-colored symmetric sequences in $\sC$} is the diagram category
\[
\symseqcc \defn \C^{\pofcop \times \fC}.  
\]
\item
The category of \emph{$\fC$-colored sequences in $\sC$} is the diagram category
\[
\seqcc \defn \C^{\profcc} = \prod_{\profcc} \C.  
\]
\item
A typical $\fC$-colored (symmetric) sequence is also written as $\{X\duc\}$ with $X\duc \in \C$ for $d \in \fC$ and $\uc \in \profc$.
\item
Suppose $k \geq 0$.  A $\fC$-colored (symmetric) sequence $X$ is said to be \emph{concentrated at $k$} if, for $d \in \fC$ and $\uc \in \profc$,
\[
|\uc| \not= k \impliesspace X\duc = \varnothing
\]

\end{enumerate}
\end{definition}

\begin{remark}
\begin{enumerate}
\item
A $\fC$-colored (symmetric) sequence concentrated at $0$ is equivalent to a $\fC$-colored object.
\item
Since there is a coproduct decomposition
\[
\pofcop \times \fC  \cong \coprod_{d \in \fC} \coprod_{[\uc] \in \pofc} \Sigma_{[\uc]}^{\mathrm{op}} \times \{d\},
\]
there is a product decomposition
\begin{equation}
\label{csonedecomp}
\symseqcc
\cong \prod_{d\in \fC} \prod_{[\uc] \in \pofc} \C^{\Sigma_{[\uc]}^{\mathrm{op}} \times\{d\}} 
\cong \prod_{d\in \fC} \prod_{[\uc] \in \pofc} \C^{\Sigma_{[\uc]}^{\mathrm{op}}}.
\end{equation}
If $X \in \symseqcc$, then a typical component with respect to this decomposition is written as
\[
X\smallbinom{d}{[\uc]} \in \C^{\Sigma_{[\uc]}^{\mathrm{op}} \times\{d\}}.
\]
A  $\fC$-colored symmetric sequence is the $\fC$-colored version of a $1$-colored symmetric sequence \cite{harper-jpaa} (3.1).
\item
Suppose $\C$ is a cofibrantly generated model category.  Then  by the decomposition \eqref{csonedecomp} and  \cite{hirschhorn} (11.1.10 and 11.6.1), the category $\symseqcc$ inherits from $\C$ a cofibrantly generated model category structure, in which weak equivalences and fibrations are defined entrywise in $\C$.  Likewise, the category $\seqcc$ inherits from $\C$ a cofibrantly generated model category structure in which fibrations, cofibrations, and weak equivalences are all defined entrywise in $\C$.
\end{enumerate}
\end{remark}

The following colimit construction is needed to define the colored version of the circle product.

\begin{definition}
\label{tensoroverg}
Suppose $G$ is a finite non-empty connected groupoid, $X \in \C^{\gop}$, and $Y \in \C^G$.  Define
\begin{equation}
\label{tensor-over}
X \tensorover{G} Y = 
\colim \left[
\nicexy{
\gop \ar[r]^-{\Delta} & \gop \times \gop \cong \gop \times G 
\ar[r]^-{X \times Y} &
\C \times \C \ar[r]^-{\otimes} & \C
}\right].
\end{equation}
\end{definition}

We now recall from \cite{white-yau} the colored circle product, the monoids with respect to which are $\fC$-colored operads.

\begin{definition}
\label{coloredcircleproduct}
Suppose:
\begin{itemize}
\item
$A, B \in \symseqcc$, and $Y \in \ctothec$;
\item
$d \in \fC$, $\uc = (c_1,\ldots,c_m) \in \pofc$, and $[\ub] \in \pofc$ is an orbit.
\end{itemize}
\begin{enumerate}
\item
Define $Y^{\otimes [\ub]} \in \C^{\Sigma_{[\ub]}}$ as the diagram with values
\begin{equation}
\label{y-tensor-b}
Y^{\otimes [\ub]}(\ub') = Y^{\otimes \ub'} 
\defn Y_{b_1'} \otimes \cdots \otimes Y_{b_n'}
\end{equation}
if $\ub' = (b_1',\ldots,b_n') \in \sigmab$.  The structure maps in the diagram $Y^{\otimes [\ub]}$ are given by permutation of the tensor factors.
\item
Define the object $B^{\otimes \uc} \in \C^{\pofcop}$ as having the $[\ub]$-component
\begin{equation}
\label{btensorc}
B^{\otimes \uc}([\ub]) 
=
\coprod_{\substack{\{[\ub_j] \in \pofc\}_{1 \leq j \leq m} \,\mathrm{s.t.} \\
[\ub] = [(\ub_1,\ldots,\ub_m)]}} 
\Kan^{\sigmaopb} 
\left[\bigotimes_{j=1}^m B\smallbinom{c_j}{[\ub_j]}\right] \in \C^{\sigmaopb}.
\end{equation}
The left Kan extension is defined as
\[
\nicexy{
\prod_{j=1}^m \sigmaop_{[\ub_j]} 
\ar[d]_-{\mathrm{concatenation}} 
\ar[rr]^-{\prod B\binom{c_j}{-}} 
&& 
\C^{\times m} \ar[r]^-{\otimes} 
& \C \ar[d]^-{=}\\
\sigmaopb \ar[rrr]_-{\Kan^{\sigmaopb}\left[\otimes B(\vdots)\right]}^-{\mathrm{left ~Kan~ extension}} &&& \C.
}\]
\item
By allowing left permutations of $\uc$ in \eqref{btensorc}, we obtain $B^{\otimes[\uc]} \in \C^{\sigmac \times \pofcop}$ with
\begin{equation}
\label{tensorbracket}
B^{\otimes [\uc]}([\ub]) \in \C^{\Sigma_{[\uc]} \times \sigmaopb}.
\end{equation}
\item
The \emph{circle product}
\[
A \circ B \in \symseqcc
\]
is defined to have components
\begin{equation}
\label{acircleb}
(A \circ B)\smallbinom{d}{[\ub]} 
= \coprod_{[\ua] \in \pofc} 
A\smallbinom{d}{[\ua]} \tensorover{\Sigma_{[\ua]}} 
B^{\otimes [\ua]}([\ub]) \in \C^{\sigmaopb \times \{d\}}
\end{equation}
for $d \in \fC$ and orbits $[\ub] \in \pofc$, where $\tensorover{\Sigma_{[\ua]}}$ is defined in \eqref{tensor-over}.
\end{enumerate}
\end{definition}

The next definition is the non-symmetric version of the previous definition.

\begin{definition}
\label{non-sym-circle-product}
Suppose $A, B \in \seqcc$.
\begin{enumerate}
\item
For $\ua=(a_1,\ldots,a_m), \ub \in \profc$, define
\[
B^{\otimes \ua}(\ub) = 
\coprod_{\substack{\{\ub_j \in \profc\}_{1 \leq j \leq m} \\
\mathrm{s.t.}\,\ub = (\ub_1,\ldots,\ub_m)}} 
\left[\bigotimes_{j=1}^m B\smallbinom{a_j}{\ub_j}\right].
\]
\item
The \emph{non-symmetric circle product}
\[
A \circ B \in \seqcc
\]
is defined to have entries
\begin{equation}
\label{acircleb-nonsym}
(A \circ B)\smallbinom{d}{\ub} 
= \coprod_{\ua \in \profc} 
A\smallbinom{d}{\ua} \otimes 
B^{\otimes \ua}(\ub)
\end{equation}
for $d \in \fC$ and $\ub \in \profc$.
\end{enumerate}
\end{definition}

Recall that a $\fC$-colored (symmetric) sequence concentrated at $0$ is equivalent to a $\fC$-colored object.  The following observation is immediate from the definition.

\begin{lemma}
\label{sobjectatzero}
Suppose $Y$ is  a $\fC$-colored object.
\begin{enumerate}
\item
For $X \in \symseqcc$ and with $Y$ regarded as a symmetric sequence concentrated at $0$, the circle product $X \circ Y$ is also concentrated at $0$.
\item
For $X \in \seqcc$ and with $Y$ regarded as a sequence concentrated at $0$, the non-symmetric circle product $X \circ Y$ is also concentrated at $0$.
\end{enumerate}
\end{lemma}

\begin{proposition}
Suppose $\fC$ is a non-empty set of colors.
\begin{enumerate}
\item
$\left(\symseqcc, \circ, \scrI\right)$ is a monoidal category with unit $\scrI$ such that
\begin{equation}
\label{circle-unit}
\scrI \duc = 
\begin{cases}
\tensorunit & \text{if $\uc = d$},\\
\varnothing & \text{if $\uc \not= d$}
\end{cases}
\end{equation}
for $\uc \in \pofc$ and $d \in \fC$.
\item
$\left(\seqcc, \circ, \scrI\right)$ is a monoidal category.
\end{enumerate}
\end{proposition}

\begin{proof}
The first assertion is proved in \cite{white-yau}.  The second assertion is proved similarly.
\end{proof}

\begin{definition}
\label{def:colored-operad}
Suppose $\fC$ is a non-empty set of colors.
\begin{enumerate}
\item
The category $\operadsigmacc$ of \emph{$\fC$-colored operads in $\sC$} is the category of monoids in the monoidal category $\left(\symseqcc, \circ, \scrI\right)$.
\item
The category $\operadomegacc$ of \emph{$\fC$-colored non-symmetric operads in $\C$} is the category of monoids in the monoidal category $\left(\seqcc, \circ, \scrI\right)$.
\item
A \emph{colored (non-symmetric) operad in $\C$} is a $\fD$-colored (non-symmetric) operad in $\C$ for some non-empty set of colors $\fD$.
\item
Suppose $\O$ is a $\fC$-colored (non-symmetric) operad in $\C$.  The category of algebras over the monad \cite{maclane} (VI.2)
\[
\O \comp - : \C^{\fC} \to \C^{\fC}
\]
is denoted by $\algo = \algoc$, whose objects are called \emph{$\sO$-algebras}.
\end{enumerate}
\end{definition}

\begin{proposition}
\label{algebras-colimits}
Suppose $\O$ is a $\fC$-colored (non-symmetric) operad in $\C$.  
\begin{enumerate}
\item
The category $\algo$ has all small limits, which are created and preserved by the forgetful functor
\[
\nicexy{\C^{\fC} & \algo \ar[l]}
\]
\item
The category $\algo$ has all small colimits, with reflexive coequalizers and filtered colimits created and preserved by the forgetful functors.
\end{enumerate}
\end{proposition}

\begin{proof}
If $\O$ is a $\fC$-colored operad, then the assertions are proved in \cite{white-yau}.  The non-symmetric case is proved similarly.
\end{proof}

The following observation provides a way to compute $\otimes_{\Sigma_{[\uc]}}$ \eqref{tensor-over}.

\begin{lemma}
\label{tensor-over-sigmac}
Suppose:
\begin{itemize}
\item
$\uc \in \pofc$ with length $|\uc| \geq 1$, and $\{c^i\}_{1 \leq i \leq r}$ are the distinct colors that appear in $\uc$ with $c^i$ appearing $k_i \geq 1$ times.
\item
$\ub$ is in the orbit of $\uc$ (e.g., $\ub = \uc$).
\item
$A \in \sC^{\sigmaopc}$, and $W = \{W_{c^i}\}_{1 \leq i \leq r} \in \sC^{\{c^1,\ldots, c^r\}}$ (i.e., each $W_{c^i}$ is an object in $\sC$).
\end{itemize}
Then there is a natural isomorphism
\begin{equation}
\label{compute-tensor-over}
A \tensorover{\sigmac} W^{\otimes [\uc]} 
\cong
A(\ub) \tensorover{\Sigma_{k_1} \times \cdots \times \Sigma_{k_r}} W^{\otimes \ub}.
\end{equation}
Here:
\begin{itemize}
\item
$- \tensorover{\sigmac}-$ is the colimit in \eqref{tensor-over}.
\item
$W^{\otimes [\uc]}$ and $W^{\otimes \ub}$ are as in \eqref{y-tensor-b}.
\item
$\Sigma_{k_1} \times \cdots \times \Sigma_{k_r}$ acts on $A(\ub)$ from the right by permuting the $k_i$ copies of $c_i$'s among themselves for each $1 \leq i \leq r$, and likewise for $W^{\otimes \ub}$.
\end{itemize}
\end{lemma}

\begin{proof}
First note that a self-map of $\ub$ simply permutes the $k_i$ copies of $c_i$'s among themselves for each $1 \leq i \leq r$.  Suppose $\ub'$ and $\ub''$ are two objects in the orbit of $\uc$.  Then there is a \emph{unique} order-preserving map $\sigma_{(\ub',\ub'')} : \ub' \to \ub''$ in the sense  that, for each $1 \leq i \leq r$, the $j$th copy of $c^i$ (counting from left to right) in $\ub'$ is sent to the $j$th copy of $c^i$ in $\ub''$ for all $1 \leq j \leq k_i$.  Then each map $\tau : \ub' \to \ub''$ is factored as
\[
\nicexy@C+.5cm{
\ub' \ar[dr]|-{\tau} \ar[d]_-{\tau'} \ar[r]^-{\sigma_{(\ub',\ub'')}} 
& 
\ub'' \ar[d]^-{\tau''}
\\
\ub' \ar[r]_-{\sigma_{(\ub',\ub'')}} & \ub''
}\]
for some self-maps $\tau'$ of $\ub'$ and $\tau''$ of $\ub''$ that are \emph{uniquely} determined by $\tau$.  Using these order-preserving maps $\sigma_{(\ub',\ub'')}$, it follows that the natural map 
\[
A(\ub) \tensorover{\Sigma_{k_1} \times \cdots \times \Sigma_{k_r}} W^{\otimes \ub}
\to
A \tensorover{\sigmac} W^{\otimes [\uc]} 
\]
is an isomorphism.
\end{proof}

\section{Existence of Right Bousfield Localization}
\label{sec:right-bous}

Here we first recall a few definitions and existence result regarding right Bousfield localization, taken from \cite{hirschhorn}. Let $\M$ be a model category and $K$ a set of cofibrant objects in $\M$. Let $\map(-,-)$ denote the homotopy function complex in $\M$. We describe the right Bousfield localization $\rkm$ of $\M$ with respect to $K$ (the chosen set of cells for $\rkm$).

\begin{definition}
A map $f:X\to Y$ is a \textit{$K$-colocal equivalence} if for every $A\in K$, the induced map $f_*:\map(A,X)\to \map(A,Y)$ is a weak equivalence. These maps will become the weak equivalences in $\rkm$.
\end{definition}

\begin{remark}
If $\M$ is a simplicial model category then one can use the simplicial mapping space instead of the homotopy function complex. For general $\M$, one often needs framings. However, in this paper we have avoided the need for framings by proving in our examples of interest (all of which are \textit{closed} symmetric monoidal categories) that one can use the internal Hom in place of $\map$ above.
\end{remark}

\begin{remark}
If $K$ is not a set of cofibrant objects, one can still define $\rkm$, but it agrees with $R_{K'}(\M)$ where $K'$ is a set of cofibrant replacements for every $A\in K$. Thus, it is no loss to assume $K$ consists of cofibrant objects from the start.
\end{remark}

\begin{definition}
An object $W$ is \textit{$K$-colocal} if $W$ is cofibrant and for every $K$-colocal equivalence $f:X\to Y$, the induced map $f_*:\map(W,X)\to \map(W,Y)$ is a weak equivalence.
\end{definition}

\begin{definition}
The \textit{right Bousfield localization} of $\M$ with respect to $K$ is a model structure $\rkm$ on the underlying category of $\M$, whose weak equivalences are the $K$-colocal equivalences, whose fibrations are the same as those in $\M$, and whose cofibrations are defined via the left lifting property.
\end{definition}

This model structure $\rkm$ need not exist in general. The following is a distillation of Theorem 5.1.1 in \cite{hirschhorn} (with corrections from the errata). Recall that $\M$ is \textit{right proper} if pullbacks of weak equivalences along fibrations are weak equivalences.

\begin{theorem}[Hirschhorn]
Let $\M$ be a right proper, cellular model category and $K$ a set of objects. Then $\rkm$ exists and the cofibrant objects are the $K$-colocal objects of $\M$. If every object of $\M$ is fibrant then $\rkm$ is cofibrantly generated.
\end{theorem}

Following \cite{christensen-isaksen}, we will avoid the need to assume $\M$ is cellular at any point. 

\begin{definition}
\label{def:right-localizable}
Suppose $\M$ is a model category.  We say that $\M$ is \emph{right localizable} if it satisfies the following conditions.
\begin{enumerate}
\item
$\M$ is right proper; i.e., pullbacks along fibrations preserve weak equivalences.
\item
There exists a set $\sJ$ of generating trivial cofibrations \cite{hirschhorn} (11.1.2(2)).  This means that:
\begin{enumerate}
\item
$\sJ$ permits the small object argument (i.e., the domains of the maps in $\sJ$ are $\sJ$-small);
\item
a map is a fibration precisely when it is right orthogonal to $\sJ$.
\end{enumerate}
\item
Every object in $\M$ is cofibration-small, i.e., small relative to the subclass of cofibrations.
\end{enumerate}
\end{definition}

Note that the definition of right localizable is only about the model category and is not about any set of objects in it.  Next we define weaker conditions on the model category that involve a set of cofibrant objects.  The essence of the next definition is that, in order to construct the right Bousfield localization, one needs a functorial factorization of each map into a $K$-colocal cofibration followed by a $K$-colocal trivial fibration.  Not surprisingly, this comes down to Quillen's small object argument.  The following definitions provide two different ways to use the small object argument to obtain the desired factorization, as discussed in \cite{hirschhorn} (5.2.3) and \cite{christensen-isaksen} (2.5).

\begin{definition}
\label{def:weakly-right-localizable}
Suppose $\M$ is a right proper model category, and $K$ is a set of cofibrant objects in $\M$.  
\begin{enumerate}
\item
We say that the pair $(\M,K)$ is \emph{type 1 right localizable} if it satisfies the following conditions.
\begin{enumerate}
\item
There exists a set $\sJ$ of generating trivial cofibrations.
\item
Define the sets
\begin{equation}
\label{lambdak}
\begin{split}
\Lambda(K) &= \Bigl\{\bA \otimes \ddeltan \to \bA \otimes \deltan ~|~ A \in K,\, n \geq 0 \Bigr\},\\
\Lambdakbar &= \Lambda(K) \cup \sJ,
\end{split}
\end{equation}
where $\bA$ is a cosimplicial resolution of $A \in K$ \cite{hirschhorn} (16.1.2(1)). Then $\Lambdakbar$ permits the small object argument, i.e., the domains of the maps in $\Lambdakbar$ are $\Lambdakbar$-small.
\end{enumerate}
\item
We say that the pair $(\M,K)$ is \emph{type 2 right localizable} if it satisfies the following conditions.
\begin{enumerate}
\item
There exists a regular cardinal $\kappa$ such that each object of $K$ is $\kappa$-small relative to cofibrations.
\item
Suppose given a diagram
\[
\nicexy{
X_0 \ar[r] \ar[drr] 
& X_1 \ar[r] \ar[dr] 
& \cdots \ar[r] 
& X_{\beta} \ar[r] \ar[dl] 
& \cdots\\
&& Y &&}
\]
in $\M$ in which the top row is a $\kappa$-sequence (with $\kappa$ as in the previous condition) of cofibrations, and the map $X_\beta \to Y$ is a fibration for each successor ordinal $\beta$.  Then the map $\colim_{\beta} X_{\beta} \to Y$ is also a fibration. 
\end{enumerate}
\item
We say that the pair $(\M,K)$ is \emph{right localizable} if it is either type 1 right localizable or type 2 right localizable. 
\end{enumerate}
\end{definition}

The definition of type 1 right localizable is extracted from \cite{hirschhorn} (Chapter 5), while the definition of type 2 right localizable is precisely \cite{christensen-isaksen} (2.4)

\begin{proposition}
Suppose $\M$ is a model category, and $K$ is a set of cofibrant objects.  If $\M$ is right localizable, then $(\M,K)$ is both type 1 right localizable and type 2 right localizable.  In particular, $(\M,K)$ is right localizable.
\end{proposition}

\begin{proof}
By assumption $\M$ is right proper and has a set $\sJ$ of generating trivial cofibrations, and every object in $\M$ is cofibration-small.  That $\Lambdakbar$ is $\Lambdakbar$-small follows from the fact that every map in $\Lambdakbar$ is a cofibration in $\M$ (see the proof of \cite{hirschhorn} 5.2.5) and the fact that $K$-colocal cofibrations are, in particular, cofibrations in $\M$.  So the pair $(\M,K)$ is type 1 right localizable.

To see that $(\M,K)$ is type 2 right localizable, as mentioned in \cite{christensen-isaksen} (under 2.4), just choose $\kappa$ such that the domains in $\sJ$ and the objects in $K$ are $\kappa$-small relative to cofibrations.
\end{proof}

The following general existence result of right Bousfield localization is proved in \cite{hirschhorn} (Chapter 5, in particular 5.1.1 and 5.1.2) and, with a variation of the argument, in \cite{christensen-isaksen} (Section 2).

\begin{theorem}
\label{rkm-exists}
Suppose $\M$ is a model category, and $K$ is a set of cofibrant objects such that $(\M,K)$ is right localizable.  Then the following statements hold.
\begin{enumerate}
\item
The right Bousfield localization $\rkm$ exists \cite{hirschhorn} (3.3.1).
\item
The cofibrant objects in $\rkm$ are precisely the $K$-colocal objects in $\M$.
\item
$\rkm$ is right proper.
\item
If every object in $\M$ is fibrant and if $\M$ is cofibrantly generated by $(\sI,\sJ)$, then $\rkm$ is cofibrantly generated by $\left(\Lambdakbar, \sJ\right)$.
\item
If $\M$ is a simplicial model category, then $\rkm$ inherits a simplicial model category structure.
\end{enumerate}
\end{theorem}

Graphically, the above conditions are related as follows:
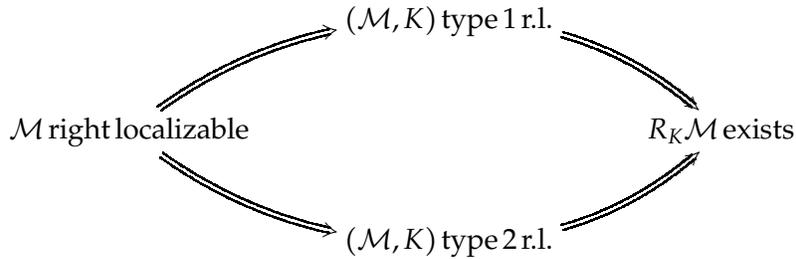
\[
\nicexy{
& (\M,K)\, \mathrm{type\, 1\, r.l.} \ar@/^1pc/@{=>}[dr] &\\
\M\, \mathrm{right\, localizable} \ar@/^1pc/@{=>}[ur] \ar@/_1pc/@{=>}[dr]
&& \rkm\, \mathrm{ exists}\\
& (\M,K)\, \mathrm{type\, 2\, r.l.} \ar@/_1pc/@{=>}[ur] &
}\]
However, to ensure that $\rkm$ is cofibrantly generated, one needs to assume a little bit more.

\section{Monoidality in Right Bousfield Localization}
\label{sec:monoidal}

In this section, we address the question of when $\rkm$ is a \textit{monoidal} model category, i.e., satisfies the pushout product axiom.

\begin{definition} \label{def:monoidal-right-bous}
A right Bousfield localization $R_K$ is said to be a \textit{monoidal right Bousfield localization} if $\rkm$ satisfies the pushout product axiom. The \textit{smallest monoidal right Bousfield localization} for a given $(\M,K)$ is a monoidal model structure $\rkm^{\otimes}$ on $\M$ such that the identity id$:\M \to \rkm^{\otimes}$ is the initial monoidal right Quillen functor to a monoidal model category taking $K$-colocal equivalences to weak equivalences
\end{definition}

This universal property is the monoidal analogue of \cite{hirschhorn} (3.3.18) and the dual of \cite{white-localization} (4.8). 

\begin{definition}
\label{def:rkm-ppa}
Suppose:
\begin{itemize}
\item
$\M$ is a model category with a set $\sJ$ of generating trivial cofibrations.
\item
$K$ is a set of cofibrant objects in $\M$.
\end{itemize}
Define the following condition.
\begin{quote}
$\$$ : If $A$ is a domain or a codomain of a map in $\Lambdakbar = \Lambdak \cup \sJ$ \eqref{lambdak}, then the functor
\[
\nicexy@C+.7cm{\M \ar[r]^-{\Hom(A,-)} & \M}
\]
takes $K$-colocal equivalences between fibrant objects to $K$-colocal equivalences.
\end{quote}
We will say that \emph{$(\M,K)$ satisfies \$} if this condition holds.
\end{definition}

Recall once again that $\M$ and $\rkm$, if it exists, have the same fibrations and also the same trivial cofibrations.  In particular, in $\$$ and everywhere else it is not necessary to specify whether an object is fibrant in $\M$ or fibrant in $\rkm$.

\begin{theorem}
\label{rkm-ppa}
Suppose:
\begin{itemize}
\item
$\M$ is a cofibrantly generated monoidal model category in which every object is fibrant.
\item
$K$ is a set of cofibrant objects in $\M$ such that $(\M,K)$ is right localizable.
\item
All the domains and codomains of the maps in $\Lambdakbar$ are $K$-colocal objects.
\end{itemize}
Then $\rkm$ is a right proper, cofibrantly generated, monoidal model category if and only if $\$$ holds.
\end{theorem}

\begin{remark}
This theorem demonstrates that the smallest monoidal Bousfield localization for a given right-localizable $(\M,K)$ has colocal objects generated by $K' = K \cup D$, where $D$ is the set of domains and codomains of maps in $\Lambdakbar$, and has weak equivalences the closure of the $K$-colocal equivalences under $\Hom(A,-)$ for $A \in K'$. Remark \ref{remark:monoidal-right-top} demonstrates this for the example of topological spaces. As remarked in \cite{gutierrez-transfer-quillen}, the existence of this model structure can be verified by carrying through the exposition in \cite{hirschhorn} or in \cite{barwickSemi} using Hom instead of $\map$, and Theorem 2.12 in \cite{gutierrez-roitzheim} provides an alternative construction in case $\M$ is combinatorial.
\end{remark}

\begin{proof}[Proof of Theorem \ref{rkm-ppa}]
First note that by Theorem \ref{rkm-exists}, $\rkm$ exists and is a right proper, cofibrantly generated model category with generating cofibrations $\Lambdakbar$.  So we must show that $\rkm$ satisfies the pushout product axiom if and only if $\$$ holds.

For the \textit{if} direction, suppose that $\$$ holds, $j:A\to B$ is a cofibration in $\rkm$, and $p:X\to Y$ is a fibration.  We write $\Hom$ for the internal hom in $\M$.  We must show that the pullback corner map $(j,p)$ in 
\begin{equation}
\label{pullback-jp}
\nicexy@C+.5cm{\hom(B,X) \ar[dr]|-{(j,p)} \ar@/^1pc/[drr]^-{\Hom(B,p)} \ar@/_1pc/[ddr] & & \\
& PB \ar[r]^-{\psharp} \ar[d] & \hom(B,Y)\ar[d] \\
& \hom(A,X) \ar[r]^-{\Hom(A,p)} & \hom(A,Y)}
\end{equation}
is a fibration that is also a $K$-colocal equivalence if either $j$ or $p$ is such.   Here
\[
PB = \Hom(A,X) \timesover{\Hom(A,Y)} \Hom(B,Y)
\]
is the pullback.
\begin{enumerate}
\item
Observe that $\rkm$ has fewer cofibrations than $\M$, so $j$ is also a cofibration in $\M$.  So $(j,p)$ is a fibration by the pushout product axiom on $\M$.
\item
We turn next to the case where $j$ is a trivial cofibration in $\rkm$, hence also in $\M$ because $\M$ and $\rkm$ have the same trivial cofibrations.   By the pushout product axiom on $\M$, $(j,p)$ is a trivial fibration in $\M$, hence also in $\rkm$.
\item
Lastly, suppose $p$ is a trivial fibration in $\rkm$. We already know that $(j,p)$ is a fibration by our first case above.  So by \cite{hovey} (4.2.5) it is sufficient to check that $(j,p)$ is a $K$-colocal equivalence for $j \in \Lambdakbar$, the set of generating cofibrations of $\rkm$.  By the $2$-out-of-$3$ property of $K$-colocal equivalences \cite{hirschhorn} (3.2.3(2)) it is enough to check that both $\psharp$ and $\Hom(B,p)$ in \eqref{pullback-jp} are $K$-colocal equivalences.

Since $A$ is the domain of $j \in \Lambdakbar$, it is by assumption a $K$-colocal object, in particular a cofibrant object in $\M$.  Recall that $\M$ and $\rkm$ have the same fibrations, and a trivial fibration in $\M$ is also a trivial fibration in $\rkm$.  So the pushout product axiom on $\M$ implies that
\begin{equation}
\label{homa-to-rkm}
\nicexy@C+10pt{
\rkm \ar@<2pt>[r]^-{-\otimes A} & \M \ar@<2pt>[l]^-{\Hom(A,-)}}
\end{equation}
is a Quillen pair because the right adjoint $\Hom(A,-)$ preserves fibrations and trivial fibrations.

Next note that, by the $2$-out-of-$3$ property of $K$-colocal equivalences, a fibrant approximation in $\M$ to a $K$-colocal equivalence is a $K$-colocal equivalence between fibrant objects.  The assumption $\$$ says that $\Hom(A,-)$ in \eqref{homa-to-rkm} takes $K$-colocal equivalences between fibrant objects to $K$-colocal equivalences.  So 3.3.18(2) in \cite{hirschhorn} now says that 
\[
\nicexy@C+10pt{\rkm \ar[r]^-{\hom(A,-)} &\rkm}
\]
is a right Quillen functor. Thus, the bottom horizontal map $\Hom(A,p)$ in \eqref{pullback-jp} is a trivial fibration in $\rkm$ because $p$ is a trivial fibration in $\rkm$.   It follows that its pullback $\psharp$ is also a trivial fibration in $\rkm$, hence a $K$-colocal equivalence.  

Likewise, since $B$ is the codomain of $j \in \Lambdakbar$, the map $\hom(B,p)$ in \eqref{pullback-jp} is a trivial fibration in $\rkm$, hence a $K$-colocal equivalence, because
\[
\nicexy@C+10pt{\rkm \ar[r]^-{\hom(B,-)} &\rkm}
\]
is a right Quillen functor and $p$ is a trivial fibration in $\rkm$. Thus, by the $2$-out-of-$3$ property in the top triangle, the pullback corner map $(j,p)$ is a $K$-colocal equivalence as required.
\end{enumerate}
We have shown that $\rkm$ satisfies the pushout product axiom.

Next, for the \textit{only if} direction, suppose $\rkm$ satisfies the pushout product axiom.  To prove $\$$, suppose $A$ is the domain or the codomain of a map in $\Lambdakbar$.  Then $A$ is cofibrant in $\rkm$ by assumption.  The pushout product axiom on $\rkm$ implies that
\[
\nicexy@C+10pt{\rkm \ar[r]^-{\hom(A,-)} &\rkm}
\]
preserves trivial fibrations in $\rkm$.  By Ken Brown's Lemma (\cite{hovey} 1.1.12), this functor takes all weak equivalences in $\rkm$ between fibrant objects to weak equivalences in $\rkm$.  This is the condition $\$$.
\end{proof}

The last assumption in Theorem \ref{rkm-ppa} says that all the domains and codomains of the maps in $\Lambdakbar$ are $K$-colocal objects.  We next address the question of when this condition holds.

\begin{theorem}
\label{rkm-tractable}
Suppose:
\begin{itemize}
\item
$\M$ is a cofibrantly generated, monoidal, \underline{simplicial} model category in which every object is fibrant.
\item
$K$ is a set of cofibrant objects in $\M$ such that $(\M,K)$ is right localizable.
\end{itemize}
Then the following statements are equivalent.
\begin{enumerate}
\item
All the generating trivial cofibrations in $\M$ have $K$-colocal domains.
\item
All the domains and codomains of the maps in $\Lambdakbar$ are $K$-colocal objects.
\item
The generating cofibrations and generating trivial cofibrations in $\rkm$ have $K$-colocal domains .
\end{enumerate}
\end{theorem}

\begin{proof}
By Theorem \ref{rkm-exists}, the right Bousfield localization $\rkm$ is a right proper, simplicial model category and is cofibrantly generated by $(\Lambdakbar,\sJ)$, where $\sJ$ is the set of generating trivial cofibrations in $\M$.  We will prove (2) $\Longrightarrow$ (3) $\Longrightarrow$ (1) $\Longrightarrow$ (2).

(2) $\Longrightarrow$ (3).  If (2) is true, then all the maps in $\sJ \subseteq \Lambdakbar$ also have $K$-colocal domains.

(3) $\Longrightarrow$ (1). If (3) is true, then all the maps in $\sJ$ ($=$ generating trivial cofibrations in $\rkm$) have $K$-colocal domains.

(1) $\Longrightarrow$ (2). Suppose (1) is true, so all the maps in $\sJ$ have $K$-colocal domains.  To prove (2), since every map in $\Lambdakbar$ is a cofibration in $\rkm$, it is enough to show that the domain of each map in $\Lambdakbar$ is $K$-colocal.  Moreover, since $\Lambdakbar = \Lambdak \cup \sJ$, it remains to show that maps in $\Lambdak$ have $K$-colocal domains.  In other words, we must show that each $\bA \otimes \ddeltan$ is a $K$-colocal object for $A \in K$.

For each $A \in K$, recall that $\bA$ is a cosimplicial resolution of $A$ in $\Mdelta$.  Using the simplicial structure on $\M$, we now choose its cosimplicial resolution as follows:
\[
\bA = \Bigl\{\bA^n = A \otimes \deltan : n \geq 0\Bigr\}.
\]
This $\bA$ is indeed a cosimplicial resolution of $A$ in $\Mdelta$ by \cite{hirschhorn} (16.1.4(1)) because $A$ is cofibrant in $\M$ by assumption.  But note that $A$ is also $K$-colocal ($=$ cofibrant in $\rkm$) and that $\rkm$ inherits its simplicial model structure from $\M$.  So by \cite{hirschhorn} (16.1.4(1)) applied to $\rkm$, the same $\bA$ is also a cosimplicial resolution of $A$ in $\rkmdelta$.  In particular, $\bA \in \rkmdelta$ is Reedy cofibrant.  So \cite{hirschhorn} (16.3.9(1)) applied to $\rkm$ says that $\bA \otimes \ddeltan$ is a cofibrant object in $\rkm$, hence $K$-colocal as desired.
\end{proof}

\begin{corollary}
\label{rkm-monoidal}
Suppose:
\begin{itemize}
\item
$\M$ is a cofibrantly generated, monoidal, simplicial model category in which every object is fibrant.
\item
$K$ is a set of cofibrant objects in $\M$ such that $(\M,K)$ is right localizable.
\item
All the generating trivial cofibrations in $\M$ have $K$-colocal domains.
\end{itemize}
Then $\rkm$ is a right proper, cofibrantly generated, monoidal, simplicial model category if and only if $\$$ holds.
\end{corollary}

\begin{proof}
This follows from Theorems \ref{rkm-exists}, \ref{rkm-ppa}, and \ref{rkm-tractable}.
\end{proof}

\section{Model Structure on Algebras in Right Bousfield Localization}
\label{sec:model-on-algebras}

In this section we provide sufficient conditions under which all colored operads in $\rkm$ are admissible.

The following definitions and examples are from \cite{bm03,bm07}.

\begin{definition}
\label{def:interval}
Suppose $\M$ is a model category that is also a symmetric monoidal closed category with $\otimes$-unit $I$.
\begin{enumerate}
\item
A \emph{coalgebra interval} consists of a counital comonoid object $H$ in $\M$ together with a factorization
\begin{equation}
\label{coalgebra-interval}
\nicexy{I \amalg I \ar[r]^-{(i_0, i_1)} & H \ar[r]^-{\varepsilon} & I}
\end{equation}
of the fold map of $I$, in which:
\begin{itemize}
\item
both maps $(i_0,i_1)$ and $\varepsilon$ are maps of comonoids;
\item
$(i_0,i_1)$ is a cofibration;
\item
$\varepsilon$ is a weak equivalence.
\end{itemize}
\item
A \emph{cocommutative coalgebra interval} is a coalgebra interval $H$ whose comonoid structure is cocommutative.
\end{enumerate}
\end{definition}

\begin{example}
In the context of the previous definition:
\begin{enumerate}
\item
Familiar examples of monoidal model categories admitting a cocommutative coalgebra interval include \cite{bm03} the categories of compactly generated Hausdorff spaces, simplicial sets, simplicial (pre)sheaves, and symmetric spectra. Less familiar examples include small categories (by analogy with spaces) and cochain complexes concentrated in non-negative degrees over a commutative ring with unit \cite{richter} (Section 8).
\item
The category of unbounded chain complexes over a commutative ring with unit admits a coalgebra interval that is \emph{not} cocommutative \cite{bm03} (3.3.3).
\end{enumerate}
\end{example}

\begin{definition}
Suppose $\M$ is a model category, and $T$ is a monad on $\M$.
\begin{enumerate}
\item
A \emph{fibrant $T$-algebra} in $\M$ is a $T$-algebra whose underlying object is fibrant in $\M$.
\item
For $X \in \algtm$, a \emph{path object} in $\M$ is a factorization
\begin{equation}
\label{path-object}
\nicexy{X \ar[r]^-{\sim} & \Path(X) \ar@{>>}[r] & X \times X}
\end{equation}
in $\algtm$ of the diagonal map in which the first map is a  weak equivalence in $\M$ and the second map is a fibration in $\M$.
\end{enumerate}
\end{definition}

Taking internal hom out of a coalgebra interval leads to functorial path objects just as in \cite{bm03} (proof of 3.1) and \cite{bm07} (proof of 2.1):

\begin{lemma}
\label{interval-path-object}
Suppose $\M$ is a monoidal model category with a cofibrant $\otimes$-unit and a coalgebra interval.
\begin{enumerate}
\item
If $\sO$ is a $\fC$-colored non-symmetric operad in $\M$, then fibrant $\sO$-algebras in $\M$ have functorial path objects in $\M$.
\item
If the coalgebra interval is furthermore cocommutative and if $\sO$ is a $\fC$-colored operad in $\M$, then fibrant $\sO$-algebras in $\M$ have functorial path objects in $\M$.
\end{enumerate}
\end{lemma}

\begin{proposition}
\label{path-object-rkm}
Suppose:
\begin{itemize}
\item
$\M$ is a monoidal model category with a cofibrant $\otimes$-unit and a coalgebra interval.
\item
$K$ is a class of cofibrant objects in $\M$ such that $\rkm$ exists.
\end{itemize}
Then the following statements hold.
\begin{enumerate}
\item
If $\sO$ is a $\fC$-colored non-symmetric operad in $\M$, then fibrant $\sO$-algebras in $\rkm$ have functorial path objects in $\rkm$.
\item
If the coalgebra interval is furthermore cocommutative and if $\sO$ is a $\fC$-colored operad in $\M$, then fibrant $\sO$-algebras in $\rkm$ have functorial path objects in $\rkm$.
 \end{enumerate}
\end{proposition}

\begin{proof}
Suppose $\sO$ is either a $\fC$-colored non-symmetric operad or a $\fC$-colored operad in $\M$, and $X$ is a fibrant $\sO$-algebra in $\rkm$.  This means that $X$ is an $\sO$-algebra with each colored entry fibrant in $\rkm$, i.e., fibrant in $\M$.  So $X$ is also a fibrant $\sO$-algebra in $\M$.  By Lemma \ref{interval-path-object} it has a path object \eqref{path-object} in $\M$ that is functorial in $X$.  In this path object, the first map is a weak equivalence in $\M$, hence also a weak equivalence in $\rkm$ \cite{hirschhorn} (3.3.3(2)(a)).  The second map is a fibration in $\M$, hence also a fibration in $\rkm$ by definition.  So this is actually a path object in $\rkm$.
\end{proof}

\begin{remark}
We are \emph{not} asserting that the $\otimes$-unit is cofibrant in $\rkm$ or that a coalgebra interval in $\M$ is also one in $\rkm$.  The reason is that not every cofibration in $\M$ is a cofibration in $\rkm$.  Luckily, what we truly need is a functorial path object for fibrant algebras, and fibrations in $\M$ and $\rkm$ are the same.
\end{remark}

The following general transfer result is a slight modification of \cite{bm07} (2.1).    It is a colored variant of \cite{bm03} (3.2 and remark afterwards), which in turn is essentially a consequence of \cite{ss} (2.3(2)).  The proof is basically the same as in \cite{bm07}, where it is assumed that there be a (cocommutative) coalgebra interval.  Here we assume directly that fibrant algebras have functorial path objects.

\begin{theorem}
\label{berger-moerdijk}
Suppose:
\begin{itemize}
\item
$\M$ is a cofibrantly generated monoidal model category with a symmetric monoidal fibrant replacement functor. 
\item
$\sO$ is either a  $\fC$-colored operad in $\M$ or a $\fC$-colored non-symmetric operad in $\M$.
\item
fibrant $\sO$-algebras have functorial path objects.
\item 
the domains of maps in $\sO \comp \sI$ (resp. $\sO \comp \sJ$) are small with respect to relative $(\sO \comp \sI)$-cell complexes (resp. relative $(\sO \comp \sJ)$-cell complexes).
\end{itemize}
Then the following statements hold.
\begin{enumerate}
\item
$\algom$ admits a cofibrantly generated model structure with weak equivalences and fibrations defined entrywise in $\M$.
\item
The sets of generating cofibrations and of generating trivial cofibrations in $\algom$ are $\sO \comp \sI$ and $\sO \comp \sJ$, respectively, where $(\sI,\sJ)$ are the sets of generating (trivial) cofibrations in $\Mtoc$ (see \cite{hirschhorn} 11.1.10).
\item
The free-forgetful adjunction
\begin{equation}
\label{operadalg-freeforget}
\nicexy{\Mtoc \ar@<2pt>[r]^-{\sO \comp -} & \algom \ar@<2pt>[l]^-{U}}
\end{equation}
is a Quillen adjunction.
\end{enumerate}
\end{theorem}

We now apply Theorem \ref{berger-moerdijk} to right Bousfield localization $\rkm$.

\begin{theorem}
\label{cor:strongly-preserve-operad}
Suppose $\M$ is a cofibrantly generated, monoidal, simplicial model category in which every object is fibrant.  Furthermore, suppose:
\begin{itemize}
\item
It has a cofibrant $\otimes$-unit and a coalgebra interval $H$.
\item
$K$ is a set of cofibrant objects in $\M$ such that $(\M,K)$ is right localizable.
\item
All the generating trivial cofibrations in $\M$ have $K$-colocal domains.
\item
It satisfies the condition $\$$ in Definition \ref{def:rkm-ppa}.
\end{itemize}
Then the following statements hold.
\begin{enumerate}
\item
If $\sO$ is a $\fC$-colored non-symmetric operad in $\M$ such that the domains of maps in $\sO \comp \sI_K$ (resp. $\sO \comp \sJ$) are small with respect to relative ($\sO \comp \sI_K)$-cell complexes (resp. relateive $(\sO \comp \sJ)$-cell complexes), then $\algorkm$ admits a cofibrantly generated model structure such that:
\begin{enumerate}
\item
Weak equivalences and fibrations in $\algorkm$ are defined entrywise in $\rkm$.
\item
The sets of generating cofibrations and of generating trivial cofibrations in $\algorkm$ are $\sO \comp \sI_K$ and $\sO \comp \sJ_K$, respectively, where $(\sI_K,\sJ_K)$ are the sets of generating (trivial) cofibrations in $\rkmc$.
\item
The free-forgetful adjunction
\begin{equation}
\label{free-forgetful-o-rkm}
\nicexy{(\rkm)^{\fC} \ar@<2pt>[r]^-{\sO \comp -} & \algorkm \ar@<2pt>[l]^-{U}}
\end{equation}
is a Quillen adjunction.
\end{enumerate}
\item
Suppose the coalgebra interval $H$ is cocommutative.  Then the conclusions of the previous case hold for every $\fC$-colored operad in $\M$.
\end{enumerate}
\end{theorem}

\begin{proof}
By Corollary \ref{rkm-monoidal} $\rkm$ is a right proper, cofibrantly generated, monoidal, simplicial model category in which every object is fibrant (because every object in $\M$ is fibrant).  By Proposition \ref{path-object-rkm} fibrant $\sO$-algebras in $\rkm$ have functorial path objects in $\rkm$.  Therefore, Theorem \ref{berger-moerdijk} applies to $\rkm$ to yield the desired conclusions.
\end{proof}

\begin{remark}
\label{rk:subsume}
By Lemma \ref{interval-path-object} the hypotheses of Theorem \ref{cor:strongly-preserve-operad} subsume those of Theorem  \ref{berger-moerdijk}.  So when Theorem \ref{cor:strongly-preserve-operad} is applicable, then so is Theorem  \ref{berger-moerdijk}. 
\end{remark}

Denote by $\Top$ either the category of compactly generated spaces or the category of compactly generated weak Hausdorff spaces.  For a finite group $G$, denote by $\Top^G$ the respective category of $G$-equivariant spaces and maps.

\begin{proposition} \label{prop:smallness-in-top}
If $\M$ is a combinatorial model category or if $\M$ is either $\Top$ or $\Top^G$ for a compact Lie group $G$, then the smallness hypotheses in Theorems \ref{berger-moerdijk} and \ref{cor:strongly-preserve-operad} are satisfied.
\end{proposition}

\begin{proof}
If $\M$ is combinatorial then all objects are small relative to the whole category, and the same remains true in $\algorkm$. 

For the category of ($G$-equivariant) topological spaces, Lemma 2.4.1 in \cite{hovey} implies all spaces are small relative to the class of topological inclusions.  Since $\Top$ is $\Top^{\{e\}}$ with $\{e\}$ the trivial group, it is enough to consider the case $\Top^G$ for a general compact Lie group $G$.   We use $\otimes$ to denote the monoidal product in $\Top^G$.  In the category of compactly generated spaces (resp. compactly generated weak Hausdorff spaces), inclusions (resp. closed inclusions) are saturated with respect to transfinite composition, pushout, smash product, Cartesian product, and passage to coinvariants under a groupoid action, as can readily be checked (see \cite{white-topological} for an exposition). 

It is therefore sufficient to argue that $\sO \circ \sI, \sO \circ \sJ,$ and $\sO \circ \sI_K$ are contained in the class of inclusions (resp. closed inclusions).  By the definition of the circle product \eqref{acircleb} and Lemma \ref{tensor-over-sigmac} (in the symmetric case), by the definition of the non-symmetric circle product \eqref{acircleb-nonsym}, and by the saturation properties of (closed) inclusions stated in the previous paragraph, it suffices to show that every map in $\sI$, $\sJ$, and $\sI_K$ is a (closed) inclusion.  This is certainly true for $\sI$ and $\sJ$.  Using Lemma \ref{lemma:top-internal-hom-test}, we take $\sI_K$ to be the set of maps of the form $K\otimes i_n$ where $i_n \in \sI$.  It follows that every map in $\sI_K$ is a (closed) inclusion.
\end{proof}

\begin{theorem} \label{thm:oalg-g-spaces-model}
Let $G$ be a compact Lie group, and let $\sO$ be a colored operad in $\Top^G$. 
\begin{enumerate}
\item Then $\sO$-algebras in $\Top^G$ inherit a transferred model structure from the usual (Quillen) model structure on $\Top^G$. 
\item Let $K$ be a set of cofibrant objects such that $(\Top^G,K)$ is right localizable, satisfies condition $\$$ in Definition \ref{def:rkm-ppa}, and all objects of the form $(G/H)_{+} \wedge D^{n}_{+}$ are $K$-colocal. Then  $\sO$-algebras in $R_K \Top^G$ inherit a transferred model structure from $R_K \Top^G$.
\item The same result holds for unpointed $G$-spaces, assuming in (2) that all objects of the form $(G/H) \times D^{n}$ are $K$-colocal.
\end{enumerate}
\end{theorem}

\begin{proof}
All three statements follow from Theorem \ref{cor:strongly-preserve-operad}, applied to $\M = \Top^G$ (and with trivial $K$ for the statement about $\sO$-algebras in $\Top^G$). The requisite smallness is Proposition \ref{prop:smallness-in-top}. The symmetric monoidal fibrant replacement functor is the identity. The existence of functorial path objects follows from Proposition \ref{interval-path-object}, since $G$-spaces have a cocommutative coalgebra interval (indeed, a commutative and cocommutative Hopf interval object as in \cite{bm03}) and has a cofibrant unit. The interval is $[0,1]$ with the trivial $G$-action (or $[0,1]_+$ in the pointed setting). The multiplication is the usual multiplication of integers (in the pointed setting, $a*p = p$ for all $a$, where $p$ is the distinguished basepoint). The unit map is $\eta: \ast \to [0,1]$ taking $\ast$ to 1 (in the pointed case, $\eta: S^0\to [0,1]_+$ preserves basepoints and takes the other point of $S^0$ to $1$). The comultiplication is the diagonal. The counit $\epsilon: [0,1]\to \ast$ is trivial (and basepoint preserving in the pointed case). All of these morphisms are easily seen to be equivariant, and the coherence conditions to give $[0,1]$ the structure of a cocommutative coalgebra interval are verified precisely as in the non-equivariant case \cite{bm03}.
\end{proof}

This Theorem will be used in Section \ref{sec:app-equivar-spaces} to obtain numerous examples of situations where right Bousfield localizations preserve $\sO$-algebra structure.

\begin{remark}
The situation of family model structures, for a family $\sF$ of subgroups of $G$, is more subtle. Unless $G$ itself is in the family $\sF$, it is not known that the unit is cofibrant.
\end{remark}

The category $\Sp^G$ of orthogonal $G$-spectra also has a (co)commutative Hopf interval object. It is simply $\Sigma^\infty [0,1]_+$. The multiplication and comultiplication are obtained using that $\Sigma^\infty$ is strong symmetric monoidal \cite[Lemma 1.8]{mandell-may-schwede-shipley}. The coherence conditions then follow from those for $[0,1]_+$ in $G$-spaces. However, no model structure on $G$-spectra can admit a cofibrant unit and symmetric monoidal fibrant replacement functor, by a well-known counterexample due to Gaunce Lewis (expounded in \cite[Section 11]{may-e-infty}). In order to have a model structure on $\sO$-algebras for the commutative monoid operad $\sO$, one must use a {\em positive} model structure, see \cite{white-commutative-monoids} for a discussion. Unfortunately, such model structures do not have a cofibrant unit (by design), so Proposition \ref{interval-path-object} (and also \cite[Proposition 4.1]{bm03}) cannot be used. However, \cite[Corollary 2.7]{kro} gets around the assumption that the unit is cofibrant, by using the interval object in another category $\cat V$ in which $\cat M$ is enriched. For us, $cat V$ is $\Top^G$. By mimicking Kro's approach \cite{kro}, with respect to the positive stable model structure on $G$-spectra \cite[Theorem III.5.3]{mandell-may-equivariant} and for the positive complete model structure on $G$-spectra \cite[Proposition B.63]{kervaire}, we obtain:

\begin{theorem} \label{thm:kro-g-spectra-fib-rep}
The positive stable model structure on $G$-equivariant orthogonal spectra admits a symmetric monoidal fibrant replacement functor, for any compact Lie group $G$. The same is true for the positive complete model structure.
\end{theorem}

\begin{proof}
This is the $G$-equivariant analogue of \cite[Theorem 3.3]{kro}. Thinking of a $G$-spectrum as a functor from an indexing category of $G$-representations \cite[Definition II.2.6]{mandell-may-equivariant}, one defines the symmetric monoidal fibrant replacement of a $G$-spectrum $X$ via the formula $TX(V) = \hocolim_n \limits \Omega^{­nV} X({}_{(n+1)}V )$, where $V$ is a finite dimensional real $G$-representation, and ${}_nV$ is $V \oplus \dots \oplus V$ ($n$ times). The homotopy colimit is taken in $G$-spaces, and the verification that $TX$ is a $G$-spectrum, and that the functor $T$ is symmetric monoidal, proceeds precisely as in \cite{kro}. The point is that the $G$ action never mixes with Kro's considerations. For the verification that $TX$ is fibrant for the positive stable model structure, observe that positive fibrant $G$-spectra are positive $\Omega-G$-spectra \cite[Theorem III.5.3]{mandell-may-equivariant}. Now Kro's proof \cite[Theorem 3.3]{kro} carries through to the equivariant setting, but now only $G$-representations $V$ with $\dim V^G \neq 0$ are considered. This is actually more than is needed for Kro's proof to work: Kro only needs $\dim V\neq 0$ \cite[Theorem 3.3]{kro}. Hence, the proof also works for the positive complete model structure, which considers, for all closed subgroups $H$ of $G$, the $H$-representations $V$ that contain a non-zero invariant vector. In particular, this restriction on $V$, and the identification of fibrant $G$-spectra as $\Omega$-spectra \cite[Diagram (B.68)]{kervaire}, allows Kro's proof to carry through.
\end{proof}

\begin{remark}
The category of orthogonal $G$-spectra also has a positive {\em flat} stable model structure \cite[Theorem 2.3.27]{stolz-thesis} and a complete {\em flat} stable model structure, obtained just as in \cite[Theorem 2.3.27]{stolz-thesis}, but replacing the consideration of $V$ with $\dim V^G \neq 0$ by the consideration of $H$-representations $V$ that contain a non-zero invariant vector. These model structures have fewer fibrations. We therefore do not know that fibrant objects are $\Omega$-spectra, and hence it is not clear if Kro's proof \cite{kro} works in these settings.
\end{remark}

\begin{corollary} \label{cor:oalg-g-spectra-model}
Let $G$ be a compact Lie group, and let $\sO$ be a colored operad in orthogonal $G$-spectra. 
\begin{enumerate}
\item Then $\sO$-algebras in $\Sp^G$ inherit a transferred model structure from the positive stable (or positive complete) model structure on $\Sp^G$. 
\item Let $K$ be a set of cofibrant objects such that $(\Sp^G,K)$ is right localizable, satisfies condition $\$$ in Definition \ref{def:rkm-ppa}, and all objects of the form $F_W((G/H)_{+} \wedge D^{n}_{+})$ are $K$-colocal, for every $G$-representation $W$ in the universe indexing $G$-spectra \cite[Definition III.2.2]{mandell-may-equivariant}. Then  $\sO$-algebras in $R_K \Sp^G$ inherit a transferred model structure from $R_K \Sp^G$.
\end{enumerate}
\end{corollary}

\begin{proof}
The first statement follows from \cite[Corollary 2.7]{kro}, combined with Theorem \ref{thm:kro-g-spectra-fib-rep}. The second statement follows from Theorem \ref{cor:strongly-preserve-operad}, except instead of appealing to Proposition \ref{path-object-rkm} (the only place where the cofibrant $\otimes$ unit is required), we appeal to the proof of \cite[Corollary 2.7]{kro}, which provides the required functorial path objects.
\end{proof}

\begin{remark}
In this section, we have discussed conditions on a model category $\calm$ to guarantee that operad-algebras in a right Bousfield localization $\rkm$ inherit a transferred model structure. These conditions are satisfied in many, but not all, of the examples of interest. In the next section, we will prove several results regarding preservation of algebraic structure under right Bousfield localization. It will turn out that we only need a transferred {\em semi}-model structure on operad-algebras in $\rkm$. We work in this more general setting to allow these results to hold for a wider class of examples, since transferred semi-model structures exist in great generality, as demonstrated in \cite{white-yau}.
\end{remark}

\section{Preservation of Algebras Under Right Bousfield Localization}
\label{sec:preservation}

In this section we will discuss preservation of monadic algebras under right Bousfield localization.

\begin{definition}
\label{def:preserves-algebra}
Assume that:
\begin{itemize}
\item
$\M$ is a  model category, and $T$ is a monad on $\M$.
\item
$K$ is a class of cofibrant objects in $\M$ such that the right Bousfield localization $\rkm$ exists.
\end{itemize}
Then $R_K$ is said to \textit{preserve $T$-algebras} if:
\begin{enumerate}
\item When $X$ is a $T$-algebra there is some $T$-algebra $\tilde{X}$ that is weakly equivalent in $\M$ to $R_K X$.
\item In addition, when $X$ is a fibrant $T$-algebra, there is a choice of $\tilde{X}$ in $\algtm$, with $U(\tilde{X})$ colocal in $\M$, there is a $T$-algebra homomorphism $c_X:\tilde{X}\to X$ lifting the colocalization map $q:R_K UX\to UX$ up to homotopy, and there is a weak equivalence $\beta_X:U(\tilde{X}) \to R_K UX$ such that $q\circ \beta_X \cong U c_X$ in Ho$(\M)$.
\end{enumerate}
If $T = \fO$ is the free $\sO$-algebra monad of a colored operad $\sO$, then we say that $R_K$ \textit{preserves $\sO$-algebras} if it preserves $\fO$-algebras.
\end{definition}

The following general preservation result is the right Bousfield localization analogue of Theorem 7.2.3 in \cite{white-yau}, which deals with left Bousfield localization.  Although we are mostly interested in colored operads, the following preservation result holds for general monads.  Furthermore, it only requires semi-model structures on $T$-algebras.

\begin{theorem} 
\label{right-preservation}
Suppose:
\begin{enumerate}
\item
$\M$ is a  model category, and $T$ is a monad on $\M$.
\item
$K$ is a class of cofibrant objects in $\M$ such that the right Bousfield localization $\rkm$ exists.
\item
$\alg(T;\M)$ inherits a projective semi-model structure from $\M$ with weak equivalences and fibrations created in $\M$. 
\item
$\alg(T;\rkm)$ inherits a projective semi-model structure from $\rkm$ with weak equivalences and fibrations created in $\rkm$. 
\item
The forgetful functor
\begin{equation}
\label{forget-algtrkm}
U:\alg(T;\rkm) \to \rkm
\end{equation}
preserves cofibrant objects.
\end{enumerate}
Then $R_K$ preserves $T$-algebras.
\end{theorem}

\begin{proof}
Let $Q_K$ denote cofibrant replacement in $\rkm$, let $Q_{K,T}$ denote cofibrant replacement in $\alg(T; \rkm)$, and let $Q_T$ and $B_T$ denote cofibrant replacement and fibrant replacement in $\alg(T; \M)$.  Since $\algtm$ is a \textit{semi}-model category, we will only apply $B_T$ to cofibrant objects in $\algtm$. We first focus on the first form of preservation and at the end turn our attention to the case where $E$ is a fibrant $T$-algebra.

Pick a $T$-algebra $E$. If $E$ is not cofibrant in $\algtm$, we first take its cofibrant replacement $Q_T E$ in $\algtm$.  Since $Q_T E \to E$ is a trivial fibration in $\algtm$, it is also a trivial fibration in $\M$.  Applying the functorial fibrant replacement $B$ in $\M$ then yields a weak equivalence $B Q_T E \to BE$ in $\M$.  So applying cofibrant replacement in $\rkm$ yields a weak equivalence
\[
\nicexy{Q_K B Q_T E \ar[r] & Q_K B E}
\]
in $\rkm$ between $K$-colocal objects.  Thus, \cite{hirschhorn} (3.2.13(2)) implies that it is actually a weak equivalence in $\M$.

Because $Q_K$ is the left derived functor of the identity adjunction between $\M$ and $\rkm$, and $B$ is the right derived functor of the identity, we know that $B_K E$ is weakly equivalent to $Q_K B E$ in $\M$.  Combined with the previous paragraph, we infer that $B_K E$ is weakly equivalent to $Q_K B Q_T E$ in $\M$.  Our model of $\tilde{E}$ will be $Q_{K,T} B_T Q_T E$.  We must therefore show
\[
Q_K B Q_T E \simeq Q_{K,T}B_T Q_T E
\]
in $\calm$.

The fibrant replacement $Q_T E \to B Q_T E$ is a trivial cofibration in $\M$. The fibrant replacement $Q_T E\to B_T Q_T E$--which exists because $Q_T E$ is cofibrant in $\algtm$--is a weak equivalence in $\alg(T; \M$), hence in $\M$. The map $B_T Q_T E\to \ast$ is a fibration in $\alg(T; \M)$, hence in $\M$. Consider the following lifting diagram in $\M$:
\begin{align} \label{diagram-fib-rep-lift}
\nicexy{
Q_T E \ar[r]^-{\simeq} \ar@{>->}[d]_-{\simeq} 
& B_T Q_T E \ar@{->>}[d] 
\\
B Q_T E \ar[r] \ar@{..>}[ur]^-{\alpha} & \ast}
\end{align}
The lifting axiom gives the dotted lift $\alpha : B Q_T E\to B_T Q_T E$, and it is necessarily a weak equivalence in $\M$ by the 2-out-of-3 property.

Since $B_T Q_T E$ is a $T$-algebra in $\M$ it must also be a $T$-algebra in $\rkm$. In the following diagram, the left vertical is a cofibration in $\rkm$, the right vertical is a trivial fibration in $\alg(T; \rkm)$, hence in $\rkm$, and the bottom horizontal map is a weak equivalence in $\rkm$. We may therefore construct a dotted lift $\beta$:
\begin{align*}
\nicexy{
\emptyset \ar@{>->}[d] \ar[r] 
& Q_{K,T} B_{T} Q_T E \ar@{->>}[d]^-{\simeq}
\\
Q_K B_{T} Q_T E \ar[r]^-{\simeq} \ar@{..>}[ur]^-{\beta}
& B_T Q_T E}
\end{align*}
By the 2-out-of-3 property, the lift $\beta$ is a weak equivalence in $\rkm$. We make use of this map as the horizontal map in the lower right corner of the diagram below.

The top horizontal map $\alpha : B Q_T E \to B_T Q_T E$ in the following diagram is the first map we constructed, which was proven to be a weak equivalence in $\M$. The square in the diagram below is then obtained by applying $Q_K$ to that map. In particular, the map
\[
Q_K B Q_T E \to Q_K B_T Q_T E
\]
is a weak equivalence in $\rkm$:
\begin{align*}
\nicexy{
B Q_T E \ar[r]^-{\alpha} & B_T Q_T E& 
\\ Q_K B Q_T E \ar[r]^-{Q_K\alpha} \ar[u] 
& Q_K B_T Q_T E \ar[u] \ar[r]^-{\beta} 
& Q_{K,T} B_T Q_T E}
\end{align*}
We have shown that both of the bottom horizontal maps are weak equivalences in $\rkm$. Thus, by the 2-out-of-3 property, their composite
\[
Q_K B Q_T E \to Q_{K,T} B_T Q_T E
\]
is a weak equivalence in $\rkm$. All the objects in the bottom row are cofibrant in $\rkm$.  Note that $Q_{K,T} B_T Q_T E$ is cofibrant in $\rkm$ by the assumption that the forgetful functor $\alg(T;\rkm) \to \rkm$ preserves cofibrant objects.  So the above $K$-colocal equivalences are actually weak equivalences in $\M$ by Theorem 3.2.13(2) in \cite{hirschhorn}.

As $E$ was a $T$-algebra and $B_T Q_T$ and $Q_{K,T}$ are endofunctors on categories of $T$-algebras, it is clear that $Q_{K,T} B_T Q_T E$ is a $T$-algebra. We have just shown that $B_K E$ is weakly equivalent in $\M$ to this $T$-algebra, so we are done.

We turn now to the case where $E$ is assumed to be a fibrant $T$-algebra (so $E$ is fibrant in $\M$ as well). We have seen that there is an $\M$-weak equivalence
\[
Q_K B Q_T E \to Q_{K,T} B_T Q_T E,
\]
and above we took $Q_{K,T} B_T Q_T E$ in $\M$ as our representative for $B_K E$ in $\Ho(\M)$. Because $E$ is a fibrant $T$-algebra, so is its cofibrant replacement $Q_T E$ in $\algtm$.  There are weak equivalences
\[
Q_T E \adjoint B_T Q_T E\]
in $\alg(T;\rkm)$ because all fibrant replacements of a given object are weakly equivalent, e.g., by diagram (\ref{diagram-fib-rep-lift}). So passage to $B_T Q_T E$ is unnecessary when $E$ is a fibrant $T$-algebra, and we take $\tilde{E}:=Q_{K,T} Q_T E$ as our representative for $B_K E$. The map 
$c_E$ is simply the composite of cofibrant replacement maps $Q_{K,T} Q_T E \to Q_T E \to E$. The map $\beta_E$ is defined by the following lifting diagram in $\rkm$, where the right vertical map is cofibrant replacement in $\rkm$:
\[
\nicexy@R+.5cm@C+.5cm{
\varnothing \ar@{>->}[d] \ar@{>->}[r] & R_KUE \ar@{->>}[d]^-{q}_-{\simeq}\\ 
UQ_{K,T} Q_T E \ar[r]^-{Uc_X}_-{\simeq} \ar@{.>}[ur]^-{\beta} & UE}
\]
Using that $U$ preserves cofibrant objects we conclude that $UQ_{K,T} Q_T E$ is cofibrant in $\rkm$. It easily follows that $\beta$ is a weak equivalence, since it is a $K$-colocal equivalence between $K$-colocal objects (by the 2-out-of-3 property).

The localization map $R_K E\to E$ in $\Ho(\M)$ lifts to the composite
\[
\nicexy{Q_{K,T} Q_T E \ar[r] & Q_T E \ar[r] & E}
\]
in $\M$. As both $Q_T$ and $Q_{K,T}$ are taken in $\alg(T)$, this composite map is a $T$-algebra homomorphism, as desired.
\end{proof}

\begin{remark}
Section 7 in \cite{carato} gives an alternative approach to preservation of algebraic structure under colocalization, by seeking to lift colocalization functors to the level of $T$-algebras. Their approach also uses a transferred model structure on $\algtrkm$, in \cite[Theorem 7.6]{carato}. In \cite{white-yau4}, we prove that our approach is equivalent to the approach of \cite{carato}. This dualizes an earlier result of Batanin and the first author \cite{batanin-white} that holds for left Bousfield localizations and proves that the main preservation result in \cite{white-localization} is equivalent to the preservation in \cite{carato}, in situations where both apply. 
\end{remark}

\section{Entrywise Colocal Colored Operads}
\label{sec:colocal-colored-operads}

The main observation in this section is Theorem \ref{operad-preservation-rkm}.  It provides sufficient conditions under which $R_K$ preserves $\sO$-algebras for an entrywise $K$-colocal colored operad.

Throughout this section, assume that $\M$ is a model category and that $K$ is a class of cofibrant objects in $\M$.  First recall the following condition $\clubsuit$ from \cite{white-yau} (Definition 6.2.1).

\begin{definition}
Suppose $\calm$ is a symmetric monoidal category and is a model category.  Define the following condition.
\begin{quote}
$\clubsuit$: For each $n \geq 1$ and $X \in \calm^{\sigmaop_n}$ that is cofibrant in $\calm$, the function
\[
X \tensorover{\Sigma_n} (-)^{\boxprod n} : \calm \to \calm
\]
preserves cofibrations and trivial cofibrations.
\end{quote}
The condition $\clubsuit$ for cofibrations will be referred to as $\clubcof$, and the condition for trivial cofibrations as $\clubacof$.  So
\[
\clubsuit = \clubcof + \clubacof.
\]
Write $\clubsuit^{\M}$, $\clubcof^{\M}$, and $\clubacof^{\M}$ if it is necessary to specify the category $\M$.
\end{definition}

To put a suitable model structure on the category of algebras over a colored operad in $\rkm$, we will employ Theorem 6.2.3 in \cite{white-yau} to $\rkm$, for which we need $\clubk$.  In other words, we need $\clubcofk$ and $\clubacofk$.  The following lemma deals with $\clubacofk$.

\begin{lemma}
\label{clubacof-rkm}
Suppose $\rkm$ exists.  Then $\clubacofm$ implies $\clubacofk$.
\end{lemma}

\begin{proof}
Suppose $X \in \rkmsigmanop = \msigmanop$ is cofibrant in $\rkm$.  So $X$ is also cofibrant in $\M$.  The condition $\clubacofm$ says that the map
\[
\nicexy@C+.5cm{
\M \ar[r]^-{X \otimes_{\Sigma_n} (-)^{\boxprod n}}
& \M
}\]
preserves trivial cofibrations in $\M$.  But $\M$ and $\rkm$ have the same fibrations, hence the same trivial cofibrations.
\end{proof}

The following condition will be used to obtain $\clubcofk$.

\begin{definition}
\label{def:diamond-condition}
Suppose $\M$ is a symmetric monoidal category and is a model category, and $K$ is a set of cofibrant objects in $\M$.  Define the following condition.
\begin{quote}
$\diamond$ : For each $n \geq 1$, $X \in \calm^{\sigmaop_n}$ that is a $K$-colocal object in $\calm$, and map $\beta : U \to V \in \msigmanop$ that is both a fibration in $\M$ and a $K$-colocal equivalence, every solid-arrow diagram
\[
\nicexy{
& U \ar[d]^-{\beta}\\
X \ar@{.>}[ur] \ar[r] & V
}\]
in $\msigmanop$ admits a dotted lift.
\end{quote}
\end{definition}

\begin{lemma}
\label{clubcof-rkm}
Suppose $\rkm$ is a monoidal model category.   Then $\diamond$ implies $\clubcofk$.
\end{lemma}

\begin{proof}
Suppose $X \in \rkmsigmanop$ is cofibrant in $\rkm$ (i.e., a $K$-colocal object in $\M$), $f : A \to B$ is a cofibration in $\rkm$, and $p : Y \to Z$ is a trivial fibration in $\rkm$.  The condition $\clubcofk$ says that every commutative square
\[
\nicexy{
X \tensorover{\sigman} A^{\boxn} \ar[d]_-{X \tensoroversigman f^{\boxn}} \ar[r]
& Y \ar[d]^-{p}
\\
X \tensorover{\sigman} B^{\boxn} \ar[r] \ar@{.>}[ur]^-{\alpha}
& Z
}\]
in $\rkm$ admits a dotted arrow that makes the entire diagram commutative.  A dotted filler $\alpha$ exists if and only if a dotted filler $\alpha'$ exists in the diagram
\[
\nicexy@R+.5cm{
& \Hom(B^{\boxn}, Y) \ar[d]^-{(f^{\boxn}, p)}
\\
X \ar[r] \ar@{.>}[ur]^-{\alpha'}
& \Hom(A^{\boxn}, Y) \timesover{\Hom(A^{\boxn},Z)} \Hom(B^{\boxn}, Z)
}\]
in $\rkmsigmanop$.  Assuming the condition $\diamond$, to show that the dotted filler $\alpha'$ exists, we just need to see that the pullback corner map $(f^{\boxn}, p)$ satisfies the condition for $\beta$ in $\diamond$.  In other words, we need to show that the map $(f^{\boxn}, p)$ is both a fibration in $\M$ and a $K$-colocal equivalence.

Since $f$ is a cofibration in $\rkm$, the iterated pushout product $f^{\boxn}$ is also a cofibration in $\rkm$ by the pushout product axiom in $\rkm$.  The map $p$ is a trivial fibration in $\rkm$ by assumption.  So the pushout product axiom in $\rkm$ once again implies that the map $(f^{\boxn}, p)$ is a trivial fibration in $\rkm$.  Since the fibrations in $\M$ and $\rkm$ are the same, the map $(f^{\boxn}, p)$ is a fibration in $\M$.  Finally, the map $(f^{\boxn}, p)$ is a weak equivalence in $\rkm$, i.e., a $K$-colocal equivalence. 
\end{proof}

\begin{theorem}
\label{rkm-algebra-model}
Suppose:
\begin{itemize}
\item
$\rkm$ is a cofibrantly generated monoidal model category.
\item
$\clubacofm$ and $\diamond$ are satisfied.
\item
$\sO$ is a $\fC$-colored operad in $\M$ that is entrywise a $K$-colocal object in $\M$.
\end{itemize}
Then:
\begin{enumerate}
\item
The category $\alg(\sO; \rkm)$ admits a cofibrantly generated \emph{semi}-model structure over $\rkm$ such that the weak equivalences and fibrations are created in $\rkm$.
\item
If $j : A \to B \in \alg(\sO; \rkm)$ is a cofibration with $A$ cofibrant in $\alg(\sO; \rkm)$, then the underlying map of $j$ is entrywise a cofibration in $\rkm$.
\item
Every $\sO$-algebra that is cofibrant in $\alg(\sO; \rkm)$ is entrywise cofibrant in $\rkm$.
\end{enumerate}
\end{theorem}

\begin{proof}
The operad $\sO$ is entrywise cofibrant in $\rkm$ by assumption.  By Lemmas \ref{clubacof-rkm} and \ref{clubcof-rkm} the conditions $\clubacofm$ and $\diamond$ together imply $\clubk = \clubcofk + \clubacofk$.  So Theorem 6.2.3 in \cite{white-yau} can be applied to $\rkm$, and the conclusions are the three statements above.
\end{proof}

\begin{theorem}
\label{operad-preservation-rkm}
Suppose:
\begin{itemize}
\item
$\M$ and $\rkm$ are cofibrantly generated monoidal model categories.
\item
$\clubm$ and $\diamond$ are satisfied.
\item
$\sO$ is a $\fC$-colored operad in $\M$ that is entrywise a $K$-colocal object in $\M$.
\end{itemize}
Then $R_K$ preserves $\sO$-algebras.
\end{theorem}

\begin{proof}
We will use Theorem \ref{right-preservation} with $T=\fO$ the monad on $\M^{\fC}$ associated to the colored operad $\sO$ \cite{white-yau} (Definition 4.1.1).  So we now check the hypotheses in Theorem \ref{right-preservation}.
\begin{enumerate}
\item
$\rkm$ exists by assumption.
\item
Note that $\sO$ is also entrywise cofibrant in $\M$, since every cofibration in $\rkm$ is also a cofibration in $\M$ and $\sO$ is entrywise cofibrant in $\rkm$ by assumption.  So $\alg(\sO;\M)$ inherits a semi-model structure from $\M^{\fC}$ by Theorem 6.2.3 in \cite{white-yau} and the assumption that $\clubm$ is satisfied.  Likewise, $\alg(\sO;\rkm)$ inherits a semi-model structure from $\rkmc$ by Theorem \ref{rkm-algebra-model}(1).
\item
The forgetful functor
\[
\alg(\sO;\rkm) \to \rkmc
\]
preserves cofibrant objects by Theorem \ref{rkm-algebra-model}(3).
\end{enumerate}
\end{proof}

\section{Sigma Cofibrant Colored Operads}
\label{sec:cofibrant-operads}

In this section we assume a little bit more about the colored operad $\sO$ and less about $\M$ than in Theorem \ref{operad-preservation-rkm}.  The main observation is Theorem \ref{preservation-sigmacof}, which says that $R_K$ preserves $\sO$-algebras if $\sO$ satisfies a certain condition $\staro$, which is equivalent to $\sO$ being cofibrant as a colored symmetric sequence in $\rkm$.

Throughout this section, assume that $\M$ is a cofibrantly generated model category and that $K$ is a class of cofibrant objects in $\M$.   The following lemma is used below to provide necessary conditions for a colored symmetric sequence in $\rkm$ to be cofibrant.

\begin{lemma}
\label{rkm-g}
Suppose:
\begin{itemize}
\item
$\rkm$ is a cofibrantly generated model category, and $G$ is a finite connected groupoid.
\item
$f \in (\rkm)^G$ is a cofibration, where $(\rkm)^G$ has the inherited projective model structure from $\rkm$.  
\end{itemize}
Then:
\begin{enumerate}
\item
$f$ is entrywise a cofibration in $\rkm$.
\item
$f \in \M^G$ is a cofibration.
\end{enumerate}
\end{lemma}

\begin{proof}
The first statement follows from \cite{hirschhorn} (11.6.3).

For the second statement, the cofibrations in $(\rkm)^G$ are generated--as retracts of transfinite compositions of pushouts--by the set $G \cdot \sI_K$.  Here $\sI_K$ is the set of generating cofibrations in $\rkm$, while
\[
G \cdot g = \coprod_{G(-,\beta)} g
\]
for some object $\beta \in G$ with the coproduct taken in $\rkm$.  Each cofibration in $\rkm$ is also a cofibration in $\M$, so
\[
G \cdot \sI_K \subseteq G \cdot \M_{\cof}.
\]
The cofibrations in $\M$ are generated by the set $\sI$ of generating cofibrations.  Since $G \cdot -$ commutes with taking retracts, transfinite compositions, and pushouts, each cofibration in $(\rkm)^G$ is generated by $G \cdot \sI$, which is the set of generating cofibrations in $\M^G$.
\end{proof}

The following observation gives necessary conditions for a $\fC$-colored symmetric sequence in $\rkm$ to be cofibrant.

\begin{proposition}
\label{rkm-sigmacof-necessary}
Suppose $\rkm$ exists, and $X \in \symseqcrkm$ is cofibrant.  Then:
\begin{enumerate}
\item
$X$ is entrywise a $K$-colocal object in $\M$.
\item
$X \in \symseqcm$ is cofibrant.
\end{enumerate}
\end{proposition}

\begin{proof}
Since (by \cite{white-yau} (3.1.5))
\[
\symseqcrkm \cong \prod_{\singledbrc \in \sigmacopc} (\rkm)^{\sigmabrcopd},
\]
$X$ being cofibrant means that each component
\[
X\singledbrc \in (\rkm)^{\sigmabrcopd}
\]
is cofibrant.  Since each $\sigmabrcopd$ is a finite connected groupoid, Lemma \ref{rkm-g} says that each component $X\singledbrc$ is
\begin{itemize}
\item
entrywise cofibrant in $\rkm$ (i.e., a $K$-colocal object);
\item
cofibrant in $\M^{\sigmabrcopd}$.
\end{itemize}
Since we also have
\[
\symseqcm \cong \prod_{\singledbrc \in \sigmacopc} \M^{\sigmabrcopd},
\]
we conclude that $X \in \symseqcm$ is also cofibrant.
\end{proof}

\begin{definition}
\label{def:condition-star}
Suppose $X \in \symseqcm$.  Define the following condition.
\begin{quote}
$\starx$ : For each map $p : E \to B \in \symseqcm$ that is entrywise both a fibration in $\M$ and a $K$-colocal equivalence, each solid-arrow diagram
\begin{equation}
\label{dc-lift}
\nicexy{
& E\singledbrc \ar[d]^-{p}
\\
X\singledbrc \ar@{.>}[ur] \ar[r] & B\singledbrc 
}
\end{equation}
in $\M^{\sigmabrcopd}$ admits a dotted lift for each $\singledbrc \in \sigmacopc$.
\end{quote}
\end{definition}

The following observation gives a necessary and sufficient condition for a $\fC$-colored symmetric sequence in $\rkm$ to be cofibrant.

\begin{proposition}
\label{sigmacof-rkm}
Suppose $\rkm$ exists, and $X \in \symseqcrkm$.  Then $X$ is cofibrant if and only if $\starx$ holds.
\end{proposition}

\begin{proof}
Suppose
\[
p : E \to B \in \symseqcrkm  \cong \prod_{\singledbrc \in \sigmacopc} (\rkm)^{\sigmabrcopd}
\]
is a trivial fibration, i.e., each component in $(\rkm)^{\sigmabrcopd}$ is a trivial fibration.  This means precisely that $p$ is entrywise a trivial fibration in $\rkm$ by \cite{hirschhorn} (11.6.1).  In other words, $p$ is entrywise both a fibration in $\M$ (since $\M$ and $\rkm$ have the same fibrations) and a $K$-colocal equivalence ($=$ weak equivalence in $\rkm$).  The object $X$ is cofibrant if and only if for each such $p$ and each diagram
\[
\nicexy{
& E \ar[d]^p\\
X \ar[r] \ar@{.>}[ur] & B
}\]
in $\symseqcrkm$, a dotted filler exists.  Such a dotted filler exists if and only if the diagram \eqref{dc-lift} in $(\rkm)^{\sigmabrcopd} = \M^{\sigmabrcopd}$ always has a dotted filler for each $\singledbrc \in \sigmacopc$.  This is exactly $\starx$.
\end{proof}

\begin{theorem}
\label{preservation-sigmacof}
Suppose:
\begin{itemize}
\item
$\M$ and $\rkm$ are both cofibrantly generated monoidal model categories.
\item
$\sO$ is a $\fC$-colored operad in $\M$ that is cofibrant in $\symseqcrkm$.
\end{itemize}
Then $R_K$ preserves $\sO$-algebras.
\end{theorem}

\begin{proof}
As before we will use Theorem \ref{right-preservation} with $T$ the monad on $\M^{\fC}$ associated to the colored operad $\sO$.  First note that $\rkm$ exists by assumption.
\begin{enumerate}
\item
$\alg(\sO;\M)$ inherits a semi-model structure from $\M^{\fC}$ by Theorem 6.3.1 in \cite{white-yau}.  This is the case because $\sO$ being cofibrant in $\symseqcrkm$  implies $\sO \in \symseqcm$ is cofibrant (by Proposition \ref{rkm-sigmacof-necessary}(2)).
\item
Likewise, $\alg(\sO;\rkm)$ inherits a semi-model structure from $\rkmc$ by, once again,  Theorem 6.3.1 in \cite{white-yau} applied to $\rkm$ and the fact that $\sO \in \symseqcrkm$ is cofibrant.
\item
That the forgetful functor
\[
\alg(\sO;\rkm) \to \rkmc
\]
preserves cofibrant objects is part of Theorem 6.3.1 in \cite{white-yau}.
\end{enumerate}
\end{proof}

\section{Application: Spaces}
\label{sec:applications}

In this section we provide applications to right Bousfield localizations to topological spaces. For this entire section $\Top$ will denote the category of pointed compactly generated spaces, since Remark 3.1.10 of \cite{hirschhorn} demonstrates that these are the only interesting right Bousfield localizations. Everything in this section is also true for pointed compactly generated weak Hausdorff spaces, using Proposition \ref{prop:smallness-in-top} for the requisite smallness. We will often use the fact that for topological model categories like $\Top$, one can use topological (rather than simplicial) mapping spaces to detect $K$-colocal equivalences. The following is Lemma 2.3 and 2.5 in \cite{christensen-isaksen} together with Example 9.1.15 of \cite{hirschhorn}:

\begin{lemma} \label{lemma:top-internal-hom-test}
Let $K$ be a set of cofibrant objects in a topological model category $\M$ such that the pair $(\M,K$) is right localizable. Then the set
\[\Lambda(K) = 
\bigl\{A \wedge S^{n-1}_+  \to A \wedge D^n_+ \;|\; A\in K,\, n \geq 0\bigr\}
\]
together with the generating trivial cofibrations $\sJ$ of $\M$ form a set of generating cofibrations of $\rkm$.
\end{lemma}

Specializing to the case of $\M = \Top$, we find that \$ (Definition \ref{def:rkm-ppa}) is always satisfied.

\begin{theorem} \label{thm:all-monoidal-in-Top}
Let $K$ be any set of cofibrant objects in $\Top$ that are small with respect to the cofibrations. Then $(\Top,K)$ satisfies \$.
\end{theorem}

\begin{proof}
Suppose $f$ is a $K$-colocal equivalence. Let $A$ be a domain or codomain of a map in $\Lambda(K) \cup \sJ$, so $A$ is of the form $K \wedge S^j_+$, $K \wedge D^j_+$, $D^j_+$, or $(D^j \times I)_+$. First, we argue that when $A$ is contractible \$ is satisfied. We must show $\Hom(A,f)$ is a $K$-colocal equivalence, which is equivalent to $\Hom(K,\Hom(A,f))$ being a weak homotopy equivalence, where $\Hom(A,X)$ has the compact open topology. This means the following map must be an isomorphism for $t \geq 0$:
\[
[S^t,\Hom(K\wedge A,f)] \cong [S^t\wedge K,\Hom(A,f)] \cong [S^t\wedge K,f] \cong [S^t,\Hom(K,f)].
\]
The right-most map is an isomorphism because $f$ is a $K$-colocal equivalence. The middle isomorphism uses that $\Hom(A,f)\simeq f$ because $A\simeq pt$. 

Next, we argue that the class of $K$-colocal equivalences coincides with the class of $K'$-colocal equivalences, where $K'$ is the closure of $K$ under all suspensions $-\wedge S^m_+$. This is because of the following chain of equivalent statements (and it also follows from Lemma 5.5.2 in \cite{hirschhorn}, taking $L$ to be the simplicial circle):

\noindent $f$ is a $K$-colocal equivalence\\
$\Longleftrightarrow$ $\Hom(K,f)$ is a weak equivalence\\
$\Longleftrightarrow$ $[S^t,\Hom(K,f)]$ is an isomorphism for all $t$\\ $\Longleftrightarrow$ $[S^{t+m},\Hom(K,f)]$ is an isomorphism for all $t, m$\\
$\Longleftrightarrow$ $[S^t,\Hom(S^m,\Hom(K,f))]$ is an isomorphism for all $t,m$\\
$\Longleftrightarrow$ $[S^t,\Hom(K\wedge S^m,f)]$ is an isomorphism for all $t,m$\\
$\Longleftrightarrow$ $f$ is a $K'$-colocal equivalence

This implies that $K$-colocal objects are the same as $K'$-colocal objects. Observe that the $K'$-colocal objects contain all objects $A$ of the form $K_0\wedge S^j_+$ for some $K_0\in K$. We now prove that \$ holds for such objects. To see that $\Hom(A,f)$ is a $K$-colocal equivalence, consider
\[
\Hom\bigl(K,\Hom(A,f)\bigr) 
\cong \Hom\bigl(K\wedge S^j_+,\Hom(K_0,f)\bigr).
\]
The map $\Hom(K_0,f)$ is a weak equivalence (hence a $K'$-colocal equivalence) because $f$ is a $K$-colocal equivalence. Since $K\wedge S^j_+ \subseteq K'$, we see that $\Hom\bigl(K,\Hom(A,f)\bigr)$ is a weak equivalence as required.

Since smashing with a contractible object has no effect, this proves that \$ holds for $A$ of the form $K\wedge S^j_+$ or $K\wedge D^j_+$. 
\end{proof}

\begin{remark} \label{remark:monoidal-right-top}
This theorem demonstrates that if one wishes a right Bousfield localization $R_K$ to be monoidal, one may as well right localize with respect to the set $\{\Sigma^m A\}$ over all $A$ in $K$. So for Top, the smallest right Bousfield localization of Definition \ref{def:monoidal-right-bous} can be obtained simply by enlarging $K$ in this way, and since this does not change the resulting $\rkm$, we see again that every right Bousfield localization is monoidal in Top. Note that the step of introducing $K'$ demonstrates that every right Bousfield localization in $\Top$ is stable, even though $\Top$ itself is unstable.
\end{remark}

\begin{example}[$n$-connected covers] \label{ex:n-conn-covers}
Let $K = \{S^m\;|\; m>n\}$, a set of cofibrant objects that are small relative to the cofibrations. In this case $R_K(X) = CW_A(X)$ where $A=S^n$ \cite{chacholski-thesis,farjoun}. The $K$-colocal objects are $X$ with $\pi_{\leq n}(X) = 0$, and the $K$-colocal equivalences are maps $f$ with $\pi_{>n}(f)$ an isomorphism. In this case, $K' = K$ in the proof above, and when $A = S^k_+ \wedge S^j_+$ for $k>n$ then $S^m_+\wedge A \cong S^{m+k+j}_+$ is again an element in $K$, so it is automatic that $\Hom(K,\Hom(A,f))$ is a weak equivalence for any $K$-colocal equivalence $f$. Theorems \ref{thm:all-monoidal-in-Top}, \ref{rkm-ppa}, and  \ref{rkm-tractable} demonstrate that the pushout product axiom is satisfied in $R_K(\Top)$. 
\end{example}

\begin{example} \label{ex:CW_A}
Let $A$ be any CW complex. Using the same reasoning as in the previous example, the right Bousfield localization $R_K(\Top)$, where $R_K(X) = CW_A(X)$ as  studied by Chach\'{o}lski and co-authors \cite{chacholski-thesis}, \cite{chacholski-dwyer-intermont}, \cite{chacholski-parent-stanley} and Nofech \cite{nofech}, forms a  monoidal model category.
\end{example}

An example of this type of colocalization is Mike Cole's mixed model structure on $\Top$, see \cite{may-ponto-more-concise} (19.1.9). We now give examples of preservation of algebra structure under colocalization. As we do not know $\diamond$ holds for $\Top$, we will need to focus on colocalizations and colored operads $\sO$ for which $\star^\sO$ holds.

\begin{example}
Let $K = \{S^1\}$ so that $K$-colocal spaces are precisely those $X$ with $\pi_0(X) = 0$; i.e., $X$ is path connected. The $Com$ operad has $Com(j)=S^0$ for all $j$, so is not entrywise path connected. Similarly, Ass$(j) = S^0[\Sigma_j]$ is a coproduct of copies of $S^0$ so is not path connected. However, an $E_\infty$ operad $\sO$ does have path connected spaces (contractible even), so such an operad is entrywise $K$-colocal. It is easy to check that such an operad is in fact $\Sigma$-cofibrant in the $K$-colocal model structure on symmetric sequences, since the fixed-point property of $E\Sigma_n$ guarantees the existence of an equivariant lift in a lifting problem against a $K$-colocally trivial fibration. It follows that $R_K$ preserves $E_\infty$-algebras. Similarly, the spaces of the $A_\infty$-operad are path connected CW complexes (the Stasheff associahedra), so are $K$-colocal, but we cannot say that $R_K$ preserves algebras over this operad because we do not know $A_\infty$ is $\Sigma$-cofibrant in the $K$-colocal model structure.
\end{example}

We can also build operads guaranteed to satisfy our criteria.

\begin{example} \label{ex:transfer-top-QkCom}
Let $K$ be any set of cofibrant objects. For each $n$, endow $\Top^{\Sigma_n}$ with the projective $K$-colocal model structure. This is possible because $\rkm$ is cofibrantly generated. Consider the free-operad adjunction
\[
\prod_n \Top^{\Sigma_n} \adjoint Op(\Top)
\]
Fresse's Theorem 12.2.A \cite{fresse} proves that there is a transferred semi-model structure, since $\rkm$ has the pushout product axiom, by Theorem \ref{thm:all-monoidal-in-Top}. In fact, it can be made a full model structure following \cite{rezk}, since $\rkm$ has a nice path object. Let $X=Q_K(Com)$ be the cofibrant replacement of $Com$ in this model structure, as in \cite{white-commutative-monoids}. By construction this operad satisfies $\star^X$, so $R_K$ preserves algebras over this operad.
\end{example}

\section{Application: Equivariant Spaces} \label{sec:app-equivar-spaces}

In this section we provide applications to right Bousfield localizations in the model category $\Top^G$ of $G$-equivariant spaces for a compact Lie group or a  finite group $G$.

Let $G$ be a compact Lie group. Let $\M = \Top^G$ denote the category of $G$-equivariant compactly generated spaces (everything would also work if we used compactly generated weak Hausdorff spaces). Let $\Map$ denote the space of all continuous maps, endowed with the conjugation action by $G$. This is the internal Hom object. Let $\Map_G$ denote the space of $G$-equivariant maps, endowed with the compact-open topology. This gives the enrichment of $\Top^G$ in $\Top$, so we may use Lemma \ref{lemma:top-internal-hom-test} with this mapping space. The weak equivalences (resp. fibrations) of $\Top^G$ are maps $f$ such that $H$-fixed points $f^H$ are weak equivalences (resp. fibrations) for all closed subgroups $H<G$. The generating (trivial) cofibrations are 
\[\bigl\{i \wedge (G/H)_+\;|\; H<G \mbox{ closed}\bigr\}
\]
where $i$ is a generating (trivial) cofibration in $\Top$. There is an adjunction
\[
\Map_G\bigl(X \wedge (G/H)_+,f\bigr) \cong \Map(X,f^H)
\]
for any $G$-space $X$.

\begin{proposition} \label{prop:monoidal-right-bous-in-GTop}
Suppose that a set $K$ of cofibrant objects in $\Top^G$ is closed under the operations $-\wedge S^1$ and $-\wedge (G/H)_+$ for all $H<G$.
Then $(\Top^G,K)$ satisfies \$ (Definition \ref{def:rkm-ppa}).
\end{proposition}

\begin{proof}
Let $f$ be a $K$-colocal equivalence. Let $A$ be a domain or codomain of a map in $\Lambda(K)\cup \sJ$. Then $A$ is of the form $D_+ \wedge (G/H)_+$ where $D$ is in the set
\[
\bigl\{D^j,D^j\wedge I, K_0\wedge S^j,K_0\wedge D^j\bigr\}. 
\]
over all $K_0 \in K$ and all $j$. We have argued above that contractible objects do not matter to the homotopy type of an internal hom object. We must prove that the map $\Map_G(B,\Map(A,f))$ is a weak equivalence in $\Top$ for all $B\in K$. We use \cite{mandell-may-equivariant} (III.1.6):
\[
\begin{split}
\Map_G\bigl(B,\Map(D_+ \wedge (G/H)_+,f)\bigr) 
&\cong \Map\bigl(B,\Map(D_+ \wedge (G/H)_+,f)\bigr)^G\\
&\cong \Map\bigl(B,\Map(D_+ \wedge (G/H)_+,f)^G\bigr)\\
& \cong \Map\Bigl(B,\Map_G\bigl(D_+ \wedge (G/H)_+,f)\bigr)\Bigr)
\end{split}
\]
Because $D_+ \wedge (G/H)_+$ is in $K$ by hypothesis, $\Map_G(D_+\wedge (G/H)_+,f)$ is a weak equivalence in $\Top$. Forgetting the $G$-action, we are now trying to prove
\[
\Map_G(B,\Map(A,f))\cong \Top(B,g)\]
is a weak equivalence, where $g$ is a weak equivalence in $\Top$. As $B$ is cofibrant and $\Top^G$ is a topological model category, \cite{hovey} (4.2.3) proves that $\Top(B,-)$ is a right Quillen functor, hence preserves weak equivalences between fibrant objects by Ken Brown's Lemma \cite{hovey} (1.1.12). We conclude that $\Map_G\bigl(B,\Map(A,f)\bigr)$ is a weak equivalence as required.
\end{proof}

The hypotheses of this theorem can always be arranged by enlarging $K$ if necessary. We now record the equivariant analogue of the $CW_A$ functors.

\begin{example}
Suppose $A$ is a cofibrant object in $\Top^G$. Denote the \textit{smash closure} of $A$ by
\[
K(A) = \Bigl\{\Sigma^p A^{\wedge q} \wedge (G/H_1)_+ \wedge \cdots \wedge (G/H_r)_+ \;|\; p,r\geq 0,\, q \geq 1;\, \mathrm{all }\, H_k < G\Bigr\}. 
\]
Then $K(A)$ satisfies the closure properties in Proposition \ref{prop:monoidal-right-bous-in-GTop}, so $(\Top^G, K(A))$ satisfies \$.  Note that $K(A)$ is the smallest set of objects containing $A$ that satisfies the closure properties in Proposition \ref{prop:monoidal-right-bous-in-GTop}.  
\end{example}

\begin{example}
More generally, suppose $\cala$ is a non-empty set of cofibrant objects in $\Top^G$.  Then the set
\[
K(\cala) =  \Bigl\{\Sigma^p A_1 \wedge \cdots \wedge A_q \wedge (G/H_1)_+ \wedge \cdots \wedge (G/H_r)_+ \;|\; p,r\geq 0,\, q \geq 1;\, \mathrm{all }\, H_k < G; \mathrm{all}\, A_j \in \cala\Bigr\}
\]
satisfies the closure properties in Proposition \ref{prop:monoidal-right-bous-in-GTop}, so $(\Top^G, K(\cala))$ satisfies \$.  In this case, $K(\cala)$ is the smallest set of objects containing $\cala$ that satisfies the closure properties in Proposition \ref{prop:monoidal-right-bous-in-GTop}.
\end{example}

We now record several examples unique to the setting $\Top^G$. First, consider the case where $G$ is finite.

Let $\sF$ be a nonempty set of subgroups of $G$, and let
\[
K(\sF) = \bigl\{(G/H)_+ \;|\; H\in \sF\bigr\}.
\]
A $K(\sF)$-colocal equivalence is a map $f$ such that
\[
\Map_G\bigl((G/H)_+,f\bigr) \cong f^H
\]
is a weak equivalence in $\Top$ for all $H\in \sF$.
 The fibrations in $R_{K(\sF)}(\Top^G)$ are the same as in $\Top^G$, i.e. maps $f$ such that $f^H$ is a fibration in $\Top$ for all $H<G$. The generating trivial cofibrations are maps of the form
 \[
 D_+^j \wedge (G/H)_+ \to D_+^j \wedge I \wedge (G/H)_+
 \]
 for all $j\geq 0$ and all $H<G$. The generating cofibrations are these maps together with maps of the form
 \[
 S_+^{j-1}\wedge (G/H)_+\to D_+^j \wedge (G/H)_+
 \]
for $j\geq 0$ and $H\in \sF$. Denote the model category $R_{K(\sF)}(\Top^G)$ by $R_\sF(\Top^G)$.

Suppose $\sF$ is a \emph{family}, i.e., closed under conjugation, intersection, and passage to subgroup.  Then there is a model structure $\Top^\sF$ with weak equivalences (resp. fibrations) the maps $f$ such that $f^H$ is a weak equivalence (resp. fibration) in $\Top$ for $H \in \sF$. The generating (trivial) cofibrations are maps of the form $i \wedge (G/H)_+$ where $H\in \sF$ and $i$ is a generating (trivial) cofibration of $\Top$.

\begin{theorem} \label{prop:cellularizations-of-GTop}
The identity functor $\Top^\sF \to R_\sF(\Top^G)$ is a Quillen equivalence. Moreover, $\Top^\sF$ and $R_\sF(\Top^G)$ both satisfy the pushout product axiom.
\end{theorem}

\begin{proof}
First, every fibration of $R_\sF(\Top^G)$ is a fibration of $\Top^\sF$ because it is easier to satisfy $f^H$ being a fibration for all $H\in \sF$ than to satisfy it for all $H < G$. Next, these two model categories have the same weak equivalences, so the identity Quillen adjunction is a Quillen equivalence. Next, $\Top^\sF$ satisfies the pushout product axiom by Lemma 2.19 in \cite{fausk}. Fundamentally, this is because $G$ is finite and $\sF$ is closed under intersection, so $G/H_1 \times G/H_2$ with the diagonal action is cofibrant.

By Theorems \ref{rkm-ppa} and \ref{rkm-tractable}, to see that $R_\sF(\Top^G)$ satisfies the pushout product axiom, it is enough to check that \$ holds for $(\Top^G, K(\sF))$.  In the current setting, \$ requires the map
\[
\begin{split}
\Map_G\bigl((G/H)_+, \Map(A, f)\bigr) 
&\cong \Map\Bigl((G/H)_+,\Map(A,f)\Bigr)^G\\
&\cong \Map\Bigl((G/H)_+ \wedge A,f\Bigr)^G\\
&\cong \Map_G\bigl((G/H)_+ \wedge A, f\bigr)\\
&\cong \Map\bigl(A, f^H\bigr)
\end{split}\]
to be a weak equivalence whenever $f$ is a $K(\sF)$-colocal equivalence, $H \in \sF$ (so $(G/H)_+ \in K(\sF)$), and $A$ a (co)domain of a map in $\Lambda(K(\sF)) \cup \sJ$.  Since $H \in \sF$ and $f$ is a $K(\sF)$-colocal equivalence, the map $f^H$ is a weak equivalence in $\Top$ (necessarily between fibrant objects).  Moreover, since every choice of $A$ is cofibrant, $\Map(A,f^H)$ is a weak equivalence as well by Ken Brown's Lemma \cite{hovey} (1.1.12).
\end{proof}

\begin{remark}
This Proposition remains true when $G$ is a compact Lie group, but all subgroups should be closed and $\Top^\sF$ is only a monoidal model category if $\sF$ is an Illman collection, which is automatic in the case when $G$ is finite. See Lemma 2.19 in \cite{fausk} for more details.
\end{remark}

\begin{example}
When $\sF = \{e\}$ we obtain a model structure on $\Top^G$ Quillen equivalent to the \textit{coarse model structure}, where $f$ is a weak equivalence if it is so when viewed in $\Top$. However, $R_{\{e\}}(\Top^G)$ has fewer fibrations and more cofibrations, allowing equivariant cells as well as the usual cells of $\Top$.
\end{example}

\begin{example}
When $\sF = \{G\}$, $R_\sF(\Top^G)$ is a new model structure on $\Top^G$ where a map $f$ is a weak equivalence if and only if $f^G$ is a weak homotopy equivalence in $\Top$.
\end{example}

We are prepared to give examples of preservation of algebraic structure under colocalization. We can mimic Example \ref{ex:transfer-top-QkCom} to put a model structure on the category of $G$-equivariant operads transferred from the projective model structure on $\prod_n (\Top^{G})^{\Sigma_n}$. In this model structure we can take the cofibrant replacement of $Com$, and it is an equivariant $E_\infty$-operad, which we shall call $\mathscr{E}$. This operad plays an important role in the search for algebraic models for equivariant spectra. We will show a preservation result for this operad. Note, however, that this is the wrong operad to encode complete equivariant commutativity (including norms) in $G$-spectra, because it does not allow for mixing of the $G$-action with the $\Sigma_n$-action. The correct operads to study for norms are the $N_\infty$-operads of \cite{blumberg-hill}, further studied in \cite{gutierrez-white}.

\begin{example}
The colocalizations $K(F)$ above preserve algebras over the operad $\mathscr{E}$. This is because $\mathscr{E}$ is $\Sigma$-cofibrant in $\prod_n (\Top^{G})^{\Sigma_n}$ and in $\prod_n \Top^{\Sigma_n}$, after forgetting the $G$-action, so $\mathscr{E}$-algebras inherit a semi-model structure in any ($G$-)topological model category. The model $R_{K(F)}(\Top^G)$ is a topological model category, as can readily be seen by following the proof in \cite{hirschhorn} that any right Bousfield localization of a simplicial model category is simplicial. We have already shown this $R_{K(F)}$ is an \textit{enriched} colocalization in the sense of \cite{gutierrez-transfer-quillen}, so it should not come as a surprise that the resulting model category $R_{K(F)}(\Top^G)$ is also enriched. With this semi-model structure in hand, we can easily prove $R_{K(F)}$ preserves $\mathscr{E}$-algebras, following Theorem \ref{preservation-sigmacof}.
\end{example}

\begin{example}
Let $K = \{S^{n+1}\}$, so that $K$-colocal objects are $n$-connected covers. Then both the operad $\mathscr{E}$ and the $N_\infty$ operads of \cite{blumberg-hill} are objectwise $K$-colocal and $\Sigma_n$-free. They therefore satisfy $\star^\sO$, and hence their algebras are preserved by taking $n$-connected covers.
\end{example}

\section{Application: Chain Complexes}
\label{app:chain}

In this section we provide applications to right Bousfield localizations in the model category of chain complexes of modules over a commutative ring $R$.  Colocalizations in these contexts are well-studied, as examples below will demonstrate. Recall the projective model structure on $\Ch_{\geq 0}(R)$ from \cite{dwyer-spalinski}. This model structure is combinatorial, closed symmetric monoidal, and all objects are bifibrant, i.e. both fibrant and cofibrant. Recall also the projective model structure on unbounded chain complexes $\Ch(R)$ from \cite{hovey}. This model structure is combinatorial, closed symmetric monoidal, and has all objects fibrant. Lastly, recall from \cite{hess-rational-survey} the category $\Ch^{\geq 0}(k)$ of non-negatively graded cochain complexes. It is combinatorial, closed symmetric monoidal, and has all objects fibrant. There are also formally dual non-positively graded chain complexes, cochain complexes, and unbounded cochain complexes.

In all three cases let $S(n)$ be the chain complex with $R$ in degree $n$ and 0 everywhere else, and let $D(n)$ be the chain complex with $R$ in degrees $n$ and $n-1$ (or $n+1$ for cochain complexes), 0 otherwise, and the identity map as the differential $d_n$. The monoidal product $\otimes$ is defined by 
\[
(C_\bullet \otimes D_\bullet)_n = \bigoplus_{i+j=n} (C_i \otimes_R D_j)
\]
and the internal $\Hom$ is defined by
\[
\Hom(X,Y)_n = \prod_k \Hom_R(X_k,Y_{n+k})
\]
We will record several examples of colocalizations in these settings. 

\begin{example} \label{ex:CW_A-chains}
If $A$ is an object of $\M$ then colocalization with respect to $K = \{A\}$ gives CW$_A$ by analogy with Example \ref{ex:CW_A}, and this can be viewed as $A$-cellular homological algebra.
$\rkm$ is the model categorical analogue of the localizing subcategory of $\D(R)$ generated by $A$, i.e. the smallest subcategory containing $A$ and closed under retracts and coproducts. This is the subcategory generated by $A$ under homotopy colimits.
\end{example}

An example of this type of colocalization is the mixed model structure on $\Ch(R)$ from \cite{may-ponto-more-concise} (19.1.10).

\begin{example} \label{ex:n-conn-cover-chains}
Let $K = \{S(n)\}$ for some $n$. Then the $K$-colocal objects are the $X$ such that $H_{<n}(X) = 0$, and the $K$-colocal equivalences are maps $f$ such that $H_{\geq n}(f)$ is an isomorphism. The functor $R_K$ can be viewed as an $n$-connected cover. Of course, this example is a special case of the above, and demonstrates that $\rkm = R_{K'}(\M)$ for $K' = \{S(m)\;|\;m\geq n\}$.
\end{example}

\begin{example} \label{ex:loc-coloc-triangles}
To any localization in $\M$ at a set of maps $S$, we can assign a colocalization with respect to the cofibers of $S$ as in \cite{barnes-roitzheim-stable} (9.2). The resulting triangles $R_K(X)\to X\to L_S(X)$ are much studied in the theory of triangulated categories, see \cite{neeman-book}, \cite{hovey-palmieri-strickland}.
\end{example}

We learned the following two examples from Bill Dwyer.

\begin{example}
Let $R$ be the integers and $p$ be a prime number. There is a colocalization such that $X$ is colocal precisely when $H_i(X)$ is $p$-torsion for all $i$. More generally, the authors of \cite{barnes-roitzheim-stable} study the monoidal properties of the analogous colocalization of a general commutative ring $R$ and a perfect $R$-module $A$. In this setting the colocal objects are the $A$-torsion $R$-modules.
\end{example}

\begin{example}
Let $I$ be a finitely generated ideal of $R$. There is a colocalization such that $X$ is colocal precisely when for all $i$ and all $x\in H_i(X)$, there is some integer $k$ with $I^kx = 0$. This allows for the study of local cohomology at $I$ using colocalization.
\end{example}

In addition, in \cite{bauer-coloc-chain-spectra} and \cite{bauer-boardman-colocalizations}, Bauer developed machinery to lift colocalizations of chain complexes to categories of spectra, giving yet another application of colocalizations for chain complexes.

\begin{lemma}\label{lemma:chains-internal-hom-test}
Let $K$ be a set of cofibrant objects in any of our models $\M$ of (co)chain complexes
 such that the pair $(\M,K$) is right localizable. Then the set
\[
\Lambda(K) = 
\bigl\{A \otimes_R S(n-1)  \to A \otimes_R D(n) \;|\; A\in K
\bigr\}
\]
together with the generating trivial cofibrations $\sJ$ of $\M$ form a set of generating cofibrations of $\rkm$. Here $n$ runs through all degrees of complexes in $\M$.
\end{lemma}

\begin{proof}
That these maps detect fibrations between fibrant objects by lifting is Lemma 2.3 of \cite{christensen-isaksen}, and already appears in \cite{hirschhorn} (Ch. 5). To see that these maps provide a factorization system, proceed exactly as in the proof of Lemma 2.5 of \cite{christensen-isaksen}, but replacing $\Delta[n]$ everywhere by $D(n)$ and replacing $\partial \Delta[n]$ by $S(n-1)$. This proof boils down to the small object argument, which is guaranteed to stop because $\M$ is combinatorial. Indeed, all that is required in order for Christensen-Isaksen's Hypothesis 2.4 to hold is that the domains of $J$ are small relative to the cofibrations. 
\end{proof}

\begin{theorem} \label{thm:chains-all-monoidal}
Let $\M = \Ch(R)$ denote any of our categories of chain complexes.  Then every right Bousfield localization $\rkm$ is monoidal, i.e. satisfies Condition \$.
\end{theorem}

\begin{proof}
Let $f$ be a $K$-colocal equivalence, $B\in K$, and let $A$ be a domain or codomain of a map in $\Lambda(K)\cup J$. We must show that $\Hom(B,\Hom(A,f))$ is a quasi-isomorphism. Here $A$ has the form $C \otimes_R S(n), C\otimes_R D(n), D(n)$, or 0, where $C\in K$. Since tensoring with a contractible object does not change anything, we are reduced to the case $A = C$ and $A = C\otimes_R S(n)$. For $A = C$ we know that $\Hom(C,f)$ is a quasi-isomorphism because $C \in K$ and $f$ is a $K$-colocal equivalence. For $A = C\otimes_R S(n)$, proceed as in the proof of Theorem \ref{thm:all-monoidal-in-Top} and note that $R_K$ is the same as $R_{K'}$ where $K' = \{W \otimes S(n) \;|\; W\in K, n=1,2,\dots \}$ is the stabilization of $K$ under shift suspension. To see that the class of $K$-colocal equivalences coincides with the class of $K'$-colocal equivalences, use the following chain of equivalent statements:

\noindent $f$ is a $K$-colocal equivalence\\
$\Longleftrightarrow$ $\Hom(K,f)$ is a quasi-isomorphism\\
$\Longleftrightarrow$ $H_t(\Hom(K,f))$ is an isomorphism for all $t$\\ 
$\Longleftrightarrow$ $[S(t),\Hom(K,f)]_{\Ch(R)}$ is an isomorphism for all $t$\\
$\Longleftrightarrow$ $[S(t)\otimes S(m),\Hom(K,f)]_{\Ch(R)}$ is an isomorphism for all $t, m$\\
$\Longleftrightarrow$ $[S(t),\Hom(S(m),\Hom(K,f))]_{\Ch(R)}$ is an isomorphism for all $t,m$\\
$\Longleftrightarrow$ $[S(t),\Hom(K\otimes S(m),f)]_{\Ch(R)}$ is an isomorphism for all $t,m$\\
$\Longleftrightarrow$ $\Hom(K\otimes S(m),f)$ is a quasi-isomorphism for all $m$\\
$\Longleftrightarrow$ $f$ is a $K'$-colocal equivalence

This implies that $K$-colocal objects are the same as $K'$-colocal objects. Since $C\otimes_R S(n)$ is in $K'$, $C\otimes_R S(n)$ is $K$-colocal.  So $\Hom(C\otimes_R S(n), f)$ and 
\[
\Hom\Bigl(B,\Hom(C\otimes_R S(n),f)\Bigr)\]
are quasi-isomorphisms as required. 
\end{proof}

\begin{remark}
Stable and monoidal right Bousfield localizations have been studied in \cite{barnes-roitzheim-stable}. There, stability means the class of $K$-colocal equivalences is closed under suspension, while the class of $K$-colocal objects is closed under desuspension. 
Our notion is different, and appears to be satisfied more frequently.
To \cite{barnes-roitzheim-stable}, a monoidal right Bousfield localization has $K$ closed under the monoidal product. They prove that $\rkm$ is stable and monoidal if $K$ and $\M$ are stable and monoidal. 
As Theorem \ref{rkm-monoidal} and Theorem 6.2 in \cite{barnes-roitzheim-stable} are both `if and only if' results, our Theorem recovers theirs in case $\M$ is stable and $K$ is closed under desuspension.
\end{remark}

We now record the dual of Proposition 7.5 in \cite{white-localization}.

\begin{proposition}
Let $k$ be a field of characteristic zero. The only right Bousfield localizations of $\Ch(k)_{\geq 0}$ are n-connected covers as in Example \ref{ex:n-conn-cover-chains}.
\end{proposition}

\begin{proof}
Over any principal ideal domain, the homotopy type is determined by $H_*$, so this means adding weak equivalences is equivalent to nullifying some object. All objects are coproducts of spheres $S(j)$, and killing $k^2$ in degree $n$ is the same as killing $k$ in degree $n$. Thus, the colocalization is completely determined by the highest degree in which the first nullification occurs. The colocalization is therefore equivalent to $\rkm$ for $K = \{S(n)\}$ for that highest degree $n$, and kills everything below that degree. 
\end{proof}

In \cite{white-localization} an example demonstrated that a left localization of unbounded chain complexes could fail to be monoidal. There does not appear to be a corresponding example for right localizations, leading to the conclusion that it is easier for $\rkm$ to satisfy the pushout product axiom than for a left localization. However, it is more difficult for $R_K$ to preserve operad-algebra structure than for a left localization to do so, as the following example shows. 

\begin{example} \label{ex:chain-coloc-destroys}
For $\M$ any of our categories of chain complexes, there is a right Bousfield localization that destroys monoid structure. Consider the first connective cover functor $R_K$ for $K = \{S(1)\}$. The $K$-colocal objects $X$ have $H_{\leq n}(X) = 0$ and a $K$-colocal equivalence is a map $f$ inducing an isomorphism on $H_{>n}(f)$. 

This colocalization fails to preserve algebraic structure, even over a $\Sigma$-cofibrant operad in $\M$. For example, suppose $A$ is a (unital) differential graded algebra. Then $A$ has a unit map $\eta:S(0)\to A$ and a multiplication $\mu:A\otimes_R A \to A$ such that $S(0)\otimes_R A \to A\otimes_R A \to A$ is an isomorphism. However, $R_K(A)$ cannot have a DGA structure because the unit map $S(0) \to R_K(A)$ is nullhomotopic, since $H_0(R_K(A)) = 0$. It follows that $S(0)\otimes_R R_K(A) \to R_K(A)$ is nullhomotopic, and not an isomorphism.
\end{example}

\begin{remark}
In this example, the operad $A_\infty$ is not objectwise $K$-colocal. We are about to prove that this is the only way a right localization of chain complexes can fail to preserve structure. If we restrict attention to non-unital DGAs then the operad \textit{is} $K$-colocal for $K = \{S(1)\}$ so this colocalization preserves non-unital DGAs.
\end{remark}

\begin{theorem}
\label{thm:chain-over-field}
Let $k$ be a field of characteristic zero and $\M = \Ch(k)$ be any of our categories of chain complexes. Then $\M$ satisfies $\diamond$. Furthermore, for any colored operad $\sO$ valued in $\M$, $\sO$ satisfies $\star^{\sO}$ if $\sO$ is objectwise $K$-colocal.
\end{theorem}

By Theorem \ref{preservation-sigmacof} we deduce

\begin{corollary}
For any of our categories of chain complexes over a field of characteristic zero, and any set of cofibrant objects $K$, $R_K$ preserves $\sO$-algebras for any colored operad $\sO$ that is objectwise $K$-colocal.
\end{corollary}

\begin{proof}[Proof of Theorem \ref{thm:chain-over-field}]
Since $k$ has characteristic zero, all symmetric sequences are projectively cofibrant in the model category $\symseqc(\M)$ by Maschke's Theorem. This is because every module with a $\Sigma_n$-action (i.e. every $X \in \M^{\Sigma_n}$) is $\Sigma_n$-free.
The same reasoning shows that $K$-colocal objects $X$ in $(\rkm)^{\Sigma_n}$ have $\Sigma_n$-equivariant lifts against fibrations, because the characteristic zero assumption guarantees the $\Sigma_n$-equivariance is no obstacle. So $\diamond$ holds. For the second part, let $\sO$ be objectwise $K$-colocal. Then $\sO$ is projectively cofibrant in $\symseqc(\rkm)$, because each $\sO \smallbinom{d}{[\uc]}$ is projectively cofibrant in $(\rkm)^{\Sigma_n}$.
\end{proof}

\begin{example}
For any of the $n$-connected cover colocalizations of Example \ref{ex:n-conn-cover-chains}, $R_K$ preserves algebras over any $E_\infty$ operad $\sO$, when $k$ has characteristic zero. This is because all spaces $\sO(n)$ have $H_i(\sO(n)) = 0$ for all $i$, hence are $K$-colocal. Similarly, commutative differential graded algebras are preserved.
\end{example}

\begin{example}
Suppose $(\M,K)$ is right localizable and that the $K$-colocalization functor can be chosen to be lax monoidal (e.g. see \cite{gutierrez-transfer-quillen} (5.6)). Then for any colored operad $\sO$, the sequence $R_K\sO$ defined by
\[
(R_K\sO)\smallbinom{d}{[\uc]} = R_K\left(\sO \smallbinom{d}{[\uc]}\right)\]
is a colored operad over $\rkm$. By construction, this colored operad is objectwise $K$-colocal, so $R_K$ preserves algebras over $R_K\sO$.
\end{example}

Note that this example is more difficult for topological spaces, because in that setting $R_K\sO$ need not be $\Sigma$-cofibrant even if $\sO$ is.

\begin{example}
Suppose $(\M,K)$ is right localizable and satisfies \$. Consider the adjunction given by the free operad functor
\[
\prod_{n\geq 0} \M^{{\Sigma_n}} \adjoint Op(\M)
\]
Give $\M^{{\Sigma_n}}$ the projective $K$-colocal model structure. This is possible because the $K$-colocal model structure is cofibrantly generated.
Fresse's Theorem 12.2.A \cite{fresse} gives a transferred semi-model structure on $Op(\M)$, and for $\M = \Ch(k)$ this can be made into a full model structure by Theorem 6.1.1 of \cite{hinich}, using the fact that $k$ has characteristic zero. The operad $Com$ is defined to be the operad with $Com(n)=k$ for all $n$. Let $X=Q_K(Com)$ be the cofibrant replacement of $Com$ in this transferred model structure, as in \cite{white-commutative-monoids}. By construction this operad satisfies $\star^X$, and hence $R_K$ preserves $Q_K(Com)$-algebras.
\end{example}

\begin{remark}
Returning to Example \ref{ex:chain-coloc-destroys} we see that the failure of $R_K$ to preserve differential graded structure is due to the fact that both the associative operad and the $A_\infty$ operad do not have $K$-colocal spaces. Recall that $A_\infty(j) = k[\Sigma_j]$, hence has non-trivial homology. Similarly, the colored operad whose algebras are $\mathfrak{C}$-colored operads will not be preserved by $n$-connected covers for $n < |\mathfrak{C}|$.
\end{remark}

\subsection{Equivariant Chain Complexes}

The previous material can be generalized to the equivariant setting in the same way we generalized from $\Top$ to $\Top^G$. Now $G$ is a group acting on $R$ and on all $R$-modules. A cofibrantly generated, monoidal model structure can be placed on $\Ch(R)^G$ with all objects fibrant. The authors are unaware of any papers studying this model structure, let alone colocalizations therein. We believe this is a valuable example to develop intuition about equivariant spectra, an important example not included in our theory because not all objects are fibrant. Furthermore, we wonder if Bauer's work in \cite{bauer-boardman-colocalizations} and \cite{bauer-coloc-chain-spectra} could be generalized to this setting, so that equivariant colocalizations of chains would lift to the level of spectra.

\subsection{Simplicial Abelian Groups}

The category of simplicial abelian groups has a cofibrantly generated model structure \cite{quillen} in which all objects are fibrant. This category is equivalent to the category of bounded below chain complexes, by the Dold-Kan Theorem. The normalized chains functor $N$ is a natural isomorphism, compatible with the model structures, and is monoidal by \cite{schwede-shipley-equivalences} (4.1).

It follows that all our preservation results about chain complexes hold in this setting as well. In particular, every right localizable $(\M,K)$ satisfies \$, a host of examples is given above, and for any colored operad $\sO$ whose spaces are all $K$-colocal, $R_K$ preserves $\sO$ algebras.

\subsection{Cotorsion pairs} \label{subsec:cotorsion}

The first author and Daniel Bravo are currently working out the theory of abelian left and right Bousfield localization. Similarly to Section \ref{sec:monoidal}, conditions are given so that a left or right Bousfield localization of an abelian model category is again an abelian model category. By the Hovey correspondence \cite{hovey-cotorsion}, abelian model structures are in one-to-one correspondence with Hovey triples. A Hovey triple $(\cat Q, \cat W, \cat F)$, consists of classes of cofibrant, acyclic, and fibrant objects, such that $(\cat Q \cap \cat W, \cat F)$ and $(\cat Q, \cat W \cap \cat F)$ are complete cotorsion pairs. The power of this method is that it allows one to work with objects instead of morphisms. All of the model structures we have considered on chain complexes and $R$-modules can be encoded by Hovey triples. In particular, there is a Hovey triple for the projective model structure on $Ch(R)$, where the weak equivalences are the quasi-isomorphisms, all objects are fibrant, and cofibrant objects are the dg-projective chain complexes. There is also a Hovey triple for the flat model structure \cite{gillespie-2004-flat} where the weak equivalences are the quasi-isomorphisms, the fibrant objects are the dg-cotorsion complexes, and the cofibrant objects are the dg-flat complexes. This model structure is monoidal, by \cite{gillespie-qcoh} (Theorem 5.7).

Since right Bousfield localization changes the cofibrant objects in a controlled way, it is tempting to connect the flat model structure and the classical projective model structure on $Ch(R)$ by changing the flat model structure so that the new cofibrant objects are dg-projective complexes. Formally, beginning with the flat model structure on $Ch(R)$ \cite{gillespie-2004-flat} and taking $K$ to be the set $\{R\}$ (which generates the projective modules), the resulting model category $R_K(\M)$ has fibrant objects the dg-cotorsion complexes, and 
cofibrant objects the degreewise projective complexes. The weak equivalences are defined as maps that factor into a trivial cofibration followed by a trivial fibration. The class of weak equivalences is characterized by the property that $Hom(f,S^0(R))$ is a quasi-isomorphism for all $K$-colocal equivalences $f$. Equivalently, $Hom_R(\text{cone}(f),R)$ is exact. The class of $K$-colocal equivalences strictly contains the quasi-isomorphisms, so the resulting model structure is not Quillen equivalent to the usual projective model structure on chain complexes \cite[Section 2.3]{hovey}. However, it is a new setting one could use to study the interplay between the notions of projective, flat, and cotorsion modules. 

\section{Application: Stable Module Category}
\label{app:stmod}

In this section we provide applications to right Bousfield localizations in the stable module category.  The stable module category is a triangulated category of $R$-modules much studied in representation theory \cite{happel-book}. Here $R$ is a quasi-Frobenius ring (i.e. projective modules and injective modules coincide), and projective modules have been set to zero. Equivalently, two maps are homotopic if their difference factors through a projective module. The triangulated structure is given by defining, for a given $R$-module $M$, $\Omega(M)$ to be the kernel of a map from a projective onto $M$; the inverse $\Omega^{-1}$ is the cokernel of $M\to I$ for $I$ injective. If $R = k[G]$ for a field $k$ and a finite group $G$, then this category is the homotopy theory of a combinatorial model category where all objects are bifibrant \cite{hovey} (2.2). If $R$ is of the form $k[G]$ where $k$ is a principal ideal domain and $G$ a finite group, then it is the homotopy category of a combinatorial model category with all objects fibrant, the projective model structure of \cite{hovey-cotorsion} (8.6), which coincides with \cite{hovey} (2.2) if $k$ is a field. Furthermore, the projective model structure satisfies the pushout product axiom and the monoid axiom \cite{hovey-cotorsion} (9.5), with monoidal product $-\otimes_K-$ and the diagonal $G$-action. 

We focus on the case $R = k[G]$ for a field $k$, and denote the model structure on $R$-mod of \cite{hovey} (2.2) by $\M$.  
Unlike our work on chain complexes, we now do \textit{not} want $k$ to be a field of characteristic zero. If it were, then all modules would be projective, so the stable module category would be trivial. The only interesting case when $k$ is a field is the modular case, where the characteristic of $k$ divides the order of $G$. The primary operads of interest in $\M$ are the associative and commutative operads, which encode associative $R$-algebras and commutative $R$-algebras, respectively.

Colocalizations have been studied in this context in \cite{benson-iyengar-krause-stratifying-localizing}, \cite{rickard-idempotent}, \cite{benson-carlson-rickard-complexity1}, \cite{benson-carlson-rickard-complexity2}, \cite{beligiannis-reiten} (for Torsion theories), \cite{bravo-gillespie-hovey}, \cite{shamir-coloc-stmod} and \cite{groth-stovicek}, and in highly related contexts in \cite{inassaridze-coloc-thick} and \cite{bokstedt-neeman-holims} (Section 6). In this setting, there is a relationship between colocalization and localization functors as in Example \ref{ex:loc-coloc-triangles}, and it now gives a natural equivalence between such functors \cite{hovey-palmieri-strickland} (3.1.6). Additionally, these localization-colocalization pairs are intimately related to recollements.

Example 4.1 of \cite{white-localization} demonstrates that not every left Bousfield localization of $\M$ is monoidal. Thus, we do not expect every right Bousfield localization of $\M$ to be monoidal. Corollary 1.2 of \cite{benson-iyengar-krause-colocalizing} gives a bijective correspondence between monoidal left localizations and monoidal right localizations, so in fact we know there \textit{must} be an example where the pushout product axiom fails for $\rkm$. A candidate would be colocalization with respect to the cofiber of the map from Example 4.1 of \cite{white-localization}. To avoid such examples, we shall assume condition \$ for the remainder of the section.

\begin{example}
The colocalization of \cite{rickard-idempotent} (5.8) with respect to a tensor ideal subcategory, satisfies \$ by construction. See Theorem 5.19 of \cite{rickard-idempotent}.
\end{example}

\begin{theorem} \label{thm:stmod-clubsuit}
Let $R = k[G]$ where $k$ is a field and $G$ is a finite group. The model structure $\M$ of \cite{hovey} (2.2), for the stable module category, satisfies  $\clubsuit$. Since all objects are cofibrant, this implies $\M$ satisfies $\spadesuit$ from \cite{white-yau} (6.1.1), and the strong commutative monoid axiom from \cite{white-commutative-monoids}. Furthermore, for any right localizable, stable $K$ satisfying \$, $R_K(\M)$ satisfies $\clubsuit$.
\end{theorem}

\begin{proof}
Let $X$ be an $R$-module with a $\Sigma_n$-action. Cofibrations are monomorphisms, so the functor $X\otimes_{\Sigma_n}(-)^{\boxprod n}$ automatically preserves cofibrations. We must prove it preserves trivial cofibrations. For any trivial cofibration $f$, $f^{\boxprod n}:A\to B$ is a trivial cofibration. Let $C$ be the cokernel of this map, and note that $C$ is nullhomotopic. Apply the functor $X\otimes_{\Sigma_n}(-)$. Since this functor preserves projective modules, it also preserves all nullhomotopic objects of $\M$ because they are generated by the projectives. Since $X\otimes_{\Sigma_n}(-)$ is a left adjoint, $X\otimes_{\Sigma_n}C$ is the cokernel of $X\otimes_{\Sigma_n}A \to X\otimes_{\Sigma_n}B$, and so this map is a weak equivalence as required. The strong commutative monoid axiom is the special case $X = k$, the unit of $\M$.

The ``furthermore'' part follows in much the same way. First, because $\M$ is a combinatorial, stable model category in which all objects are bifibrant, we have very good control over the $K$-colocal equivalences and the generating (trivial) cofibrations of $\rkm$. The trivial cofibrations are the same as in $\M$, so are preserved by the functors $X\otimes_{\Sigma_n}(-)^{\boxprod n}$ just as above.
To see that these functors preserve $K$-colocal cofibrations, it is enough that they preserve $K$-colocal objects. This is easy to verify, since the property of being $K$-colocal is detected using homotopy function complexes.
\end{proof}

We note that in the case $R = k[G]$ for a principal ideal domain $k$, rather than a field $k$, the cofibrant objects of the projective model structure $\M$ on $R$-modules are the Gorenstein projective modules \cite{hovey-cotorsion} (8.6). We do not know if the functors $X\otimes_{\Sigma_n} (-)^{\boxprod n}$ preserve Gorenstein projectivity. The following is a formal consequence of Theorem \ref{thm:stmod-clubsuit}, though of course direct proofs are also possible.

\begin{corollary}
There are combinatorial stable model category structures on algebras over any colored operad in $\M = k[G]$-mod, wherein maps that factor through a projective module are null.  Furthermore:
\begin{enumerate}
\item
In these model categories, all objects are fibrant. In particular, the categories of associative, commutative, Lie, $A_\infty$, and $E_\infty$, and $L_\infty$-algebras all have model structures.
\item
For any right localizable, stable $K$ satisfying \$, algebras over any entrywise $K$-colocal colored operad in $\rkm$ have transferred semi-model structures. If the associative (resp. commutative) operad is entrywise colocal, then there is a full model structure on its category of algebras in $\rkm$.
\end{enumerate}
\end{corollary}

\begin{proof}
Let $\sO$ be any colored operad in $\M$. As all objects are cofibrant, $\sO$ is entrywise cofibrant, so Theorem 6.2.3 of \cite{white-yau} proves there is a transferred semi-model structure over $\M$. Since all objects in $\M$ are cofibrant, that semi-model structure is a model structure. In these model structures, a map $f$ is a weak equivalence or fibration if it is so in $\M$. The proof for $\rkm$ is similar, but we use Theorem \ref{rkm-algebra-model} to prove that the transferred semi-model structure exists. For for associative operad, we use \cite{barnes-roitzheim-stable} (6.2) to deduce the monoid axiom for $\rkm$, and the main result of \cite{ss} for the claimed model structure. For the commutative operad, we use the main result of \cite{white-commutative-monoids}, since $\rkm$ satisfies the commutative monoid axiom and the monoid axiom.
\end{proof}


\begin{example}
Let $\M = k[G]$-mod and let $K = \{k\}$ where $k$ has the trivial $G$ action. The colocalization of $(\M,K)$ is studied in \cite{beligiannis-reiten} (IV.2.7). Since $k$ is the monoidal unit, the commutative operad is objectwise $K$-colocal. Let $K'$ denote the stable, monoidal closure of $K$, following \cite{barnes-roitzheim-stable}. Then $R_{K'}$ preserves commutative monoids. Let $K''$ be the closure of $K'$ and $\{k[\Sigma_n]\;|\;n=1,2,\dots \}$. Then $R_{K''}$ preserves both commutative monoids and associative monoids. 
\end{example}

We conclude this section by pointing out several other potential applications of our main results in algebraic settings.

\begin{remark}
The paper \cite{inassaridze-coloc-thick} conducts colocalization in the model category of \cite{joachim-johnson}, for Kasparov $K$-theory. An interesting open problem is whether or not this model category is monoidal, and whether or not that colocalization preserves algebraic structure. Similarly, \cite{becker} proved that the degreewise injective model structure on Ch(R) of \cite{gillespie-degreewise} is a right Bousfield localization of the Inj-model structure of \cite{bravo-gillespie-hovey}. For rings $R$ with monoidal Inj-model structure, our work gives a new approach to studying the monoidal properties of the degreewise injective model structure. Lastly, \cite{murfet-thesis} considered a colocalization with respect to local cohomology in the category of quasi-coherent sheaves over a scheme $X$. This is a monoidal model category by \cite{gillespie-qcoh}, and one could therefore ask about preservation of algebraic structure by this colocalization. 
\end{remark}

\section{Application: Small Categories}
\label{app:small-cat}

The category $Cat$ of small categories possesses a combinatorial, simplicial, Cartesian model structure called the \textit{canonical} (or \textit{folk}) model structure \cite{rezk-folk}. The weak equivalences are the equivalences of categories, and all objects are bifibrant.
Commonly studied operads in this setting include the categorical Barrat-Eccles operad and the $A_\infty$ operad, whose algebras are $A_\infty$-categories. The path object of $Cat$ can be used to construct model structures for these categories of algebras, following \cite{rezk}. Since $Cat$ is simplicial, \$ can easily be verified as we did for topological spaces.

For any colocalization $R_K$ chosen, model structures on algebras over the Barrat-Eccles operad or the $A_\infty$-operad can be constructed as for the stable module category; since all objects are cofibrant there will be no difference between a semi-model structure over $Cat$ and a full model structure. Thus, $\clubsuit$ and $\star^X$ need not be verified, and we will have preservation results over any colored operad that is entrywise $K$-colocal.

A notion of suspension in $Cat$ has recently been invented in \cite{vicinsky-thesis}. It should therefore be possible to define $n$-connected covers as a colocalization, and one can mimic the theory of topological spaces and chain complexes to determine what types of algebraic structure are preserved. Since the Barrat-Eccles operad is contractible, such colocalizations should preserve algebras over it.

Another colocalization in this setting is $A$-cellularization for some category or set of categories $A$, by analogy with $\Top$. The description on pages 37-39 of \cite{vicinsky-thesis} give a hands-on way to do such cellularization, replacing her $T$ by some category $A$. In particular, $A$ can be chosen in such a way that either of the operads above consist of $A$-colocal spaces. 

\begin{remark}
There are several examples similar to $Cat$, where preservation results of this sort could be proven. For example, the (folk) model structure on the category of groupoids and the (folk) model structure on set-valued operads are both combinatorial, Cartesian, and have all objects fibrant. A vast generalization of all three is given by the model structure \cite{everaert-kieboom-linden-internal-cat} of categories internal to some category $\C$. In this cofibrantly generated model structure, all objects are fibrant, and a monoidal product is inherited from $\C$ (just as $Cat$ inherits the Cartesian structure from Set). All of these examples have colocalization functors and operads, so the question of preservation may be studied therein.
\end{remark}



\begin{thebibliography}{AAAAAA}
\bibitem[BR14]{barnes-roitzheim-stable}
David Barnes and Constanze Roitzheim.
\newblock Stable left and right {B}ousfield localisations.
\newblock {\em Glasg. Math. J.}, 56(1):13--42, 2014.

\bibitem[Bar10]{barwickSemi}
Clark Barwick.
\newblock On left and right model categories and left and right {B}ousfield
  localizations.
\newblock {\em Homology, Homotopy Appl.}, 12(2):245--320, 2010.

\bibitem[BB14]{batanin-berger}
Michael Batanin and Clemens Berger.
\newblock Homotopy theory for algebras over polynomial monads, preprint
  available electronically from http://arxiv.org/abs/1305.0086.
\newblock 2014.

\bibitem[Bau99]{bauer-boardman-colocalizations} F. Bauer, The {B}oardman category of spectra, chain complexes and (co-)localizations, Homology Homotopy Appl. (1), 95-116, 1999

\bibitem[Bau02]{bauer-coloc-chain-spectra} F. Bauer, Colocalizations and their realizations as spectra, Theory Appl. Categ. 10 (8) 162-179, 2002.

\bibitem[BW16]{batanin-white}
    Michael Batanin and David White.
    \newblock Left {B}ousfield Localization and {E}ilenberg-{M}oore Categories, preprint available as arXiv:1606.01537.


\bibitem[Bec14]{becker} H. Becker, Models for Singularity Categories, Adv. Math 254, 187-232, 2014.

\bibitem[BR07]{beligiannis-reiten} A. Beligiannis and I. Reiten, Homological and homotopical aspects of torsion theories, Mem. Amer. Math. Soc. 188 (883), 2007

\bibitem[BCR95]{benson-carlson-rickard-complexity1} D. Benson, J. Carlson, J. Rickard, Complexity and varieties for infinitely generated modules, Math. Proc. Cambridge Philos. Soc. 118 (2), 223-243, 1995.

\bibitem[BCR96]{benson-carlson-rickard-complexity2} D. Benson, J. Carlson, J. Rickard, Complexity and varieties for infinitely generated modules. {II}, Math. Proc. Cambridge Philos. Soc. 120 (4), 597-615, 1996.

\bibitem[BIK11]{benson-iyengar-krause-stratifying-localizing} D. Benson, S. Iyengar, H. Krause, Stratifying modular representations of finite groups, Ann. of Math. (2) 174 (3), 1643-1684, 2011.

\bibitem[BIK12]{benson-iyengar-krause-colocalizing} D. Benson, S. Iyengar, H. Krause, Colocalizing subcategories and cosupport, J. Reine Angew. Math. 673, 161-207, 2012

\bibitem[BM03]{bm03}
C. Berger and I. Moerdijk, Axiomatic homotopy theory for operads, Comment. Math. Helv. 78 (2003) 805-831.




\bibitem[BM07]{bm07}
C. Berger and I. Moerdijk, Resolution of coloured operads and rectification of homotopy algebras, Contemp. Math. 431 (2007), 31-58.



\bibitem[BH14]{blumberg-hill}
Andrew~J. Blumberg and Michael~A. Hill.
\newblock Operadic multiplications in equivariant spectra, norms, and
  transfers, preprint, http://arxiv.org/abs/1309.1750.
\newblock 2014.

\bibitem[BN93]{bokstedt-neeman-holims} M. B{\"o}kstedt and A. Neeman, Homotopy limits in triangulated categories, Compositio Math. 86 (2), 209-234, 1993


\bibitem[Bor94]{borceux}
F. Borceux, Handbook of categorical algebra 2, categories and structures, Cambridge Univ. Press, Cambridge, UK.

\bibitem[BGH14]{bravo-gillespie-hovey} D. Bravo, J. Gillespie, M. Hovey, The stable module category of a general ring, arXiv:1405.5768, 2014

\bibitem[CRT14]{carato} C. Casacuberta, O. Ravent\'{o}s, A. Tonks, Comparing Localizations across Adjunctions, arxiv 1404.7340.

\bibitem[CCS07]{castellana-crespo-scherer-conn-covers} N. Castellana, J. Crespo, and J. Scherer, On the cohomology of highly connected covers of finite {H}opf
              spaces, Adv. Math 215 (1), 250--262, 2007

\bibitem[Cha96]{chacholski-thesis}
W. Chach\'{o}lski, On the functors {$CW_A$} and {$P_A$}, Duke Math. J. 84 (3), 599-631, 1996.

\bibitem[CDI02]{chacholski-dwyer-intermont}
W. Chach\'{o}lski, W. Dwyer, M. Intermont, The {$A$}-complexity of a space, J. London Math. Soc. (2), 65 (1), 204-222, 2002.

\bibitem[CPS04]{chacholski-parent-stanley}
W. Chach\'{o}lski, P.E. Parent, D. Stanley, Cellular generators, Proc. Amer. Math. Soc. 132 (11), 3397-3409 (electronic), 2004.

\bibitem[CI04]{christensen-isaksen}
J.D. Christensen and D. Isaksen, Duality and pro-spectra, Alg. Geom. Top., 781-812, 2004.

\bibitem[DS95]{dwyer-spalinski}
W.~G. Dwyer and J.~Spali{\'n}ski.
\newblock Homotopy theories and model categories.
\newblock In {\em Handbook of algebraic topology}, pages 73--126.
  North-Holland, Amsterdam, 1995.


\bibitem[EKV05]{everaert-kieboom-linden-internal-cat} T. Everaert, R.W. Kieboom, T. Van der Linden, Model structures for homotopy of internal categories, Theory Appl. Categ. 15 (3), 66-94, 2005

\bibitem[Far96]{farjoun}
E.D. Farjoun, Cellular Spaces, Null Spaces and Homotopy Localization, Lecture Notes in Math. 1622, Springer-Verlag, Berlin, 1996.


\bibitem[Fau08]{fausk}
Halvard Fausk, \emph{Equivariant homotopy theory for pro-spectra}, Geom. Topol.
  \textbf{12} (2008), no.~1, 103--176. \MR{2377247 (2009c:55010)}



\bibitem[FMY09]{fmy}
Y. Fr\'{e}gier, M. Markl, and D. Yau, The $L_\infty$-deformation complex of diagrams of algebras, New York J. Math. 15 (2009), 353-392.


\bibitem[Fre10]{fresse}
B. Fresse, Props in model categories and homotopy invariance of structures, Georgian Math. J. 17 (2010), 79-160.

\bibitem[Gil04]{gillespie-2004-flat} J. Gillespie, The flat model structure on $Ch(R)$, Transactions of the American Mathematics Society, Volume 356, Number 8, Pages 3369-3390.

\bibitem[Gil06]{gillespie-qcoh} J. Gillespie, The Flat Model Structure on Complexes of Sheaves, Transactions of the American Mathematical Society 358 (7), 2855--2874, 2006.

\bibitem[Gil08]{gillespie-degreewise} J. Gillespie, Cotorsion pairs and degreewise homological model structures, Homol. Homotopy Appl 10 (1), 283-304, 2008.

\bibitem[Gut12]{gutierrez-transfer-quillen} J. Guti{\'e}rrez, Transfer of algebras over operads along {Q}uillen adjunctions, J. Lond. Math. Soc. (2), 86 (2), 607--625, 2012.

\bibitem[GR14]{gutierrez-roitzheim} J. Guti{\'e}rrez and C. Roitzheim, {B}ousfield localisations along {Q}uillen bifunctors and applications, arXiv:1411.0500

\bibitem[GW13]{gutierrez-white}
Javier~J. Guti{\'e}rrez and David White.
\newblock Encoding equivariant commutativity via operads, preprint available as  arXiv:1707.02130.
\newblock 2013.




\bibitem[GS14]{groth-stovicek} M. Groth and J. Stovicek, Tilting theory for trees via stable homotopy theory, arXiv:1402.6984, 2014


\bibitem[Hap88]{happel-book} D. Happel, Triangulated categories in the representation theory of finite-dimensional algebras, London Mathematical Society Lecture Note Series 119, Cambridge University Press, 1988

\bibitem[Har10a]{agt2}
J.E. Harper, Bar constructions and Quillen homology of modules over operads, Algebr. Geom. Topol., 10(1):87--136, 2010. 



\bibitem[Har10b]{harper-jpaa}
J.E. Harper, Homotopy theory of modules over operads and non-$\Sigma$ operads in monoidal model categories, J. Pure Appl. Algebra 214 (2010), 1407-1434.



\bibitem[Hess07]{hess-rational-survey} K. Hess, Rational homotopy theory: a brief introduction, Contemp. Math., 436, 175-202, 2007

\bibitem[HHR16]{kervaire}
Michael~A. Hill, Michael~J. Hopkins, and Douglas~C. Ravenel.
\newblock On the nonexistence of elements of {K}ervaire invariant one.
\newblock {\em Annals of Mathematics}, 184:1--262, 2016.

\bibitem[Hin97]{hinich}
V. Hinich, Homological algebra of homotopy algebra, Comm. Alg. 25 (1997), 3291-3323.


\bibitem[Hir03]{hirschhorn}
P.S. Hirschhorn, Model categories and their localizations, Math. Surveys and Monographs 99, Amer. Math. Soc. Providence, RI, 2003.


\bibitem[Hov99]{hovey}
M. Hovey, Model categories, Math. Surveys and Monographs 63, Amer. Math. Soc. Providence, RI, 1999.

\bibitem[Hov02]{hovey-cotorsion} M. Hovey, Cotorsion pairs, model category structures, and representation theory, Math Z. 241 (3), 553--592, 2002

\bibitem[HPS97]{hovey-palmieri-strickland}
Mark Hovey, John~H. Palmieri, and Neil~P. Strickland.
\newblock Axiomatic stable homotopy theory.
\newblock {\em Mem. Amer. Math. Soc.}, 128(610):x+114, 1997.

\bibitem[IKM12]{inassaridze-coloc-thick} H. Inassaridze, T. Kandelaki, R. Meyer, Localisation and colocalisation of triangulated categories at thick subcategories, Math. Scand. 110 (1), 59-74, 2012

\bibitem[JJ06]{joachim-johnson} M. Joachim and M. Johnson, Realizing {K}asparov's {$KK$}-theory groups as the homotopy classes of maps of a {Q}uillen model category, Contemp. Math 399, 163-197, 2006.

\bibitem[JY09]{jy1}
M.W. Johnson and D. Yau, On homotopy invariance for algebras over colored PROPs, J. Homotopy and Related Structures 4 (2009), 275-315.

\bibitem[Kro07]{kro} T.A. Kro, Model structures on operads in orthogonal spectra, Homology, Homotopy and Applications, vol. 9(2), 2007, pp.397-412.



\bibitem[Mac98]{maclane}
S. Mac Lane, Categories for the working mathematician, Grad. Texts in Math. 5, 2nd ed., Springer-Verlag, New York, 1998.

\bibitem[MMSS01]{mandell-may-schwede-shipley}
M.~A. Mandell, J.~P. May, S.~Schwede, and B.~Shipley, \emph{Model categories of
  diagram spectra}, Proc. London Math. Soc. (3) \textbf{82} (2001), no.~2,
  441--512. \MR{2001k:55025}

\bibitem[MM02]{mandell-may-equivariant}
M.~A. Mandell and J.~P. May.
\newblock Equivariant orthogonal spectra and {$S$}-modules.
\newblock {\em Mem. Amer. Math. Soc.}, 159(755):x+108, 2002.

\bibitem[MSS02]{mss}
M. Markl, S. Shnider, and J. Stasheff, Operads in Algebra, Topology and Physics, Math. Surveys and Monographs 96, Amer. Math. Soc., Providence, 2002.

\bibitem[May72]{may72}
J.P. May, The geometry of iterated loop spaces, Lecture Notes in Math. 271, Springer-Verlag, New York, 1972.


\bibitem[May97]{may97}
J.P. May, Definitions: operads, algebras and modules, Contemporary Math. 202, p.1-7, 1997.

\bibitem[May09]{may-e-infty} J.P. May, What precisely are $E_\infty$ ring spaces and $E_\infty$ ring spectra?, 
Geometry \&
Topology
Monographs 16 (2009) 215-282.

\bibitem[MP12]{may-ponto-more-concise} J.P. May and K. Ponto, More concise algebraic topology, University of Chicago Press, 2012.

\bibitem[MM97]{mcgibbon-moller} C. McGibbon and J. M{\o}ller, Connected covers and {N}eisendorfer's localization theorem, Fund. Math. 152 (3) 211--230, 1997

\bibitem[Mur07]{murfet-thesis} D. Murfet, The mock homotopy category of projectives and Grothendieck duality, PhD thesis, Australian National University, 2007. (online at www.therisingsea.org)

\bibitem[Nee01]{neeman-book} A. Neeman, Triangulated categories, Annals of Mathematics Studies 148, Princeton University Press, 2001

\bibitem[Nof99]{nofech}
A. Nofech, $A$-cellular homotopy theories, Journal of Pure and Applied Algebra 141 (3), 249-267, 1999.

\bibitem[Qui67]{quillen}
D. Quillen, Homotopical Algebra, Lecture Notes in Mathematics, Springer-Verlag, No. 43, 1967

\bibitem[Rez96]{rezk}
C.W. Rezk, Spaces of algebra structures and cohomology of operads, Ph.D. thesis, MIT, 1996.



\bibitem[Rez96]{rezk-folk} C. Rezk, A model category for categories, preprint available electronically from http://www.math.uiuc.edu/$\sim$rezk/cat-ho.dvi, 1996

\bibitem[Rez97]{rickard-idempotent} J. Rickard, Idempotent modules in the stable category, J. London Math. Soc (2), 56 (1), 149-170, 1997

\bibitem[Ric06]{richter} B. Richter, Homotopy algebras and the inverse of the normalization functor, Journal of Pure and Applied Algebra, 206 (2006) 277-321.

\bibitem[SS00]{ss}
S. Schwede and B. Shipley, Algebras and modules in monoidal model categories, Proc. London Math. Soc. 80 (2000), 491-511.

\bibitem[SS03]{schwede-shipley-equivalences} S. Schwede and B. Shipley, Equivalences of monoidal model categories, Algebr. Geom. Topol. (3) 287-334, 2003

\bibitem[Sha11]{shamir-coloc-stmod} S. Shamir, Colocalization functors in derived categories and torsion theories, Homology Homotopy Appl. 13 (1), 75-88, 2011


\bibitem[Sto11]{stolz-thesis}
Martin Stolz, \emph{Equivariant structure on smash powers of commutative ring
  spectra}, Ph.D. thesis, University of Bergen, 2011.



\bibitem[Vic15]{vicinsky-thesis} D. Vicinsky, The homotopy calculus of categories and graphs, Ph.D. thesis, available electronically from https://scholarsbank.uoregon.edu/xmlui/bitstream/handle/\\
1794/19283/Vicinsky$\_$oregon$\_$0171A$\_$11298.pdf?sequence=1, 2015



\bibitem[Whi14a]{white-commutative-monoids}
David White.
\newblock Model structures on commutative monoids in general model categories. arXiv:1403.6759.
\newblock {\em Journal of Pure and Applied Algebra}, Volume 221, Issue 12, 2017, Pages 3124-3168.

\bibitem[Whi14b]{white-localization}
D. White. Monoidal Bousfield localizations and algebras over operads, arXiv:1404.5197.
\newblock 2014.



\bibitem[Whi13]{white-topological}
David White.
\newblock A short note on smallness and topological monoids, preprint available
  electronically from http://personal.denison.edu/$\sim$whiteda/research.html.
\newblock 2013.



\bibitem[WY15]{white-yau}
D. White and D. Yau, Bousfield localization and algebras over colored operads,  {\em Applied Categorical Structures},  Volume 26, Issue 1, pp 153-203, 2018, arXiv:1503.06720.

\bibitem[WY16a]{white-yau3}
D. White and D. Yau, Homotopical adjoint lifting theorem, preprint available electronically from http://arxiv.org/abs/1606.01803.


\bibitem[WY16b]{white-yau4}
David White and Donald Yau. Right Bousfield Localization and Eilenberg-Moore Categories, available as arXiv:1609.03635.


\bibitem[Yau]{yau-colored-operad-book}
D. Yau, Colored Operads, Graduate Studies in Math., Amer. Math. Soc., Providence, RI, to appear.


\bibitem[YJ15]{jy2}
D. Yau and M.W. Johnson, A Foundation for PROPs, Algebras, and Modules, Math. Surveys and Monographs 203, Amer. Math. Soc., Providence, 2015.

\end{thebibliography}
\end{document}